\documentclass{amsart}   	% use "amsart" instead of "article" for AMSLaTeX format                		% See geometry.pdf to learn the layout options. There are lots.
\usepackage{amssymb,amsthm}
\usepackage{amsthm}
\usepackage{graphicx}
\usepackage{verbatim}
\usepackage{enumitem}
\usepackage{tikz}
\usepackage{tikz-cd}
\usetikzlibrary{calc}
\usetikzlibrary{decorations.pathmorphing}
\usepackage{hyperref}
\usepackage{adjustbox}
\usepackage{mathtools}
\usepackage[all]{xy}

\tikzset{curve/.style={settings={#1},to path={(\tikztostart)
    .. controls ($(\tikztostart)!\pv{pos}!(\tikztotarget)!\pv{height}!270:(\tikztotarget)$)
    and ($(\tikztostart)!1-\pv{pos}!(\tikztotarget)!\pv{height}!270:(\tikztotarget)$)
    .. (\tikztotarget)\tikztonodes}},
    settings/.code={\tikzset{quiver/.cd,#1}
        \def\pv##1{\pgfkeysvalueof{/tikz/quiver/##1}}},
    quiver/.cd,pos/.initial=0.35,height/.initial=0}

\tikzset{tail reversed/.code={\pgfslambdarrowsstart{tikzcd to}}}
\tikzset{2tail/.code={\pgfslambdarrowsstart{Implies[reversed]}}}
\tikzset{2tail reversed/.code={\pgfslambdarrowsstart{Implies}}}
\tikzset{no body/.style={/tikz/dash pattern=on 0 off 1mm}}

\graphicspath{ {./images/} }
\newtheorem{theorem}{Theorem}[section]
\newtheorem{proposition}[theorem]{Proposition}
\newtheorem{lemma}[theorem]{Lemma}
\newtheorem{corollary}[theorem]{Corollary}

\theoremstyle{definition}

\newtheorem{remark}[theorem]{Remark}

\newcommand{\qedno}{\null\nobreak\hfill\ensuremath{\square}} 
\newcommand{\uxa}{\ensuremath{(\underline{X},\underline{A})}}

\newcommand{\caa}{\ensuremath{(\underline{CA},\underline{A})}}
\newcommand{\cxx}{\ensuremath{(\underline{CX},\underline{X})}}
 
\newcommand{\cyy}{\ensuremath{(\underline{CY},\underline{Y})}}

\usepackage[utf8]{inputenc}

\title{Loop space decompositions of moment-angle complexes associated to flag complexes}
\author{Lewis Stanton*}
\address{School of Mathematical Sciences, University of Southampton, University Rd, Southampton, SO17 1BJ.}
\email{lrs1g18@soton.ac.uk}
\thanks{Lewis Stanton - Email: lrs1g18@soton.ac.uk.  ORCID ID: 0000-0003-4662-054X}

\subjclass[2020]{Primary 55P15, 55P35.}
\keywords{homotopy type, loop space, polyhedral product}

\begin{document}

\begin{abstract}
We prove that the loop space of the moment-angle complex associated to the $k$-skeleton of a flag complex belongs to the class $\mathcal{P}$ of spaces homotopy equivalent to a finite type product of spheres and loops on simply connected spheres. To do this, a general result showing $\mathcal{P}$ is closed under retracts is proved.
\end{abstract}

\maketitle

\section{Introduction}
\label{sec:Intro}

Polyhedral products have attracted vast attention due to their many applications across mathematics (see \cite{BBC}). A polyhedral product is a natural subspace of $\prod_{i=1}^m X_i$ defined as follows. Let $K$ be a simplicial complex on the vertex set $[m] = \{1,2,\cdots,m\}$. For $1 \leq i \leq m$, let $(X_i,A_i)$ be a pair of pointed $CW$-complexes, where $A_i$ is a pointed $CW$-subcomplex of $X_i$. Let $\uxa = \{(X_i,A_i)\}_{i=1}^m$ be the sequence of pairs. For each simplex $\sigma \in K$, let $\uxa^\sigma$ be defined by \[ \uxa^\sigma = \prod\limits_{i=1}^m Y_i \text{ where }  Y_i = \begin{cases} X_i & i \in \sigma \\ A_i & i \notin \sigma. \end{cases}\] The \textit{polyhedral product} determined by $\uxa$ and $K$ is \[\uxa^K = \bigcup\limits_{\sigma \in K} \uxa^{\sigma} \subseteq \prod\limits_{i=1}^m X_i.\] An important special case is when $(X_i,A_i) = (D^2,S^1)$ for all $i$. These polyhedral products are called \textit{moment-angle complexes}, and are denoted $\mathcal{Z}_K$. More generally, when $(X_i,A_i) = (D^n,S^{n-1})$ for $n \geq 2$ and all $i$, the polyhedral products are called \textit{generalised moment-angle complexes}. In this paper, we identify the homotopy type of the loop space of certain polyhedral products. One particular case is when $K$ is a flag complex. When $K$ is flag, certain polyhedral products give models for the classifying space of graph products of groups, implying that the loop space of these polyhedral products are graph products of groups. This geometric group theoretic framework has been generalised by Cai \cite{Ca} to consider loops on a wider class of polyhedral products associated to flag complexes. For general simplicial complexes $K$, the loop space of the corresponding moment-angle complex is related to a certain diagonal subspace arrangement \cite{D}. 

Let $\mathcal{W}$ be the full subcategory of topological spaces which are homotopy equivalent to a finite type wedge of simply connected spheres, and let $\mathcal{P}$ be the full subcategory of $H$-spaces which are homotopy equivalent to a finite type product of spheres and loops on simply connected spheres. Note that if $X \in \mathcal{P}$, by the Hopf invariant one problem \cite{Ad}, the only spheres that can appear in a product decomposition for $X$ are $S^n$ for $n \in \{1,3,7\}$, and it will be assumed that the loops on spheres $\Omega S^n$ which appear are of dimension $n \geq 2$, $n \notin \{2,4,8\}$, as when $n \in \{2,4,8\}$, there is a homotopy equivalence $\Omega S^n \simeq S^{n-1} \times \Omega S^{2n-1}$. Relations between spaces in $\mathcal{W}$ and spaces in $\mathcal{P}$ will be used frequently throughout the paper. In particular, the Hilton-Milnor theorem \cite{H,M} implies that looping sends spaces in $\mathcal{W}$ to spaces in $\mathcal{P}$, and decomposing the suspension of a product as a wedge and the James construction \cite{J} implies that suspension sends spaces in $\mathcal{P}$ to spaces in $\mathcal{W}$.

Determining the homotopy type of polyhedral products in general is difficult, but in the special case of a moment-angle complex, progress has been made in showing that certain moment-angle complexes are in $\mathcal{W}$. For example, moment-angle complexes associated with shifted complexes \cite[Theorem 1.2]{GT2}, flag complexes with chordal 1-skeleton \cite[Theorem 4.6]{GPTW}, or more generally, totally fillable simplicial complexes \cite[Corollary 7.3]{IK2} are in $\mathcal{W}$. There is a wider range of moment-angle complexes (including the aforementioned ones) for which its loop space is in $\mathcal{P}$. For example, moment-angle complexes associated to any flag complex are in $\mathcal{P}$ \cite[Corollary 7.3]{PT}. It is known that many moment-angle complexes are not in $\mathcal{W}$ due to the existence of non-trivial cup products in cohomology. For example, when $K$ is the boundary of a square, $\mathcal{Z}_K \simeq S^3 \times S^3$, and it follows that for any simplicial complex $L$ containing $K$ as a full subcomplex, $\mathcal{Z}_L$ contains non-trivial cup products in cohomology, and so $\mathcal{Z}_L \notin \mathcal{W}$.

In this paper, we specialise to the case where $K$ is the $k$-skeleton of a flag complex. In particular, we prove the following result.
\begin{theorem}
\label{loopofkskelinP}
    Let $k \geq 0$, and let $K$ be the $k$-skeleton of a flag complex on the vertex set $[m]$ and $A_1,\cdots,A_m$ be path connected $CW$-complexes such that $\Sigma A_i \in \mathcal{W}$ for all $i$. Then $\Omega \caa^{K} \in \mathcal{P}$.
\end{theorem} There are two cases of Theorem ~\ref{loopofkskelinP} which should be highlighted. The first important case is when $k$ is the dimension of $K$. While not stated in this generality, the following result recovers \cite[Corollary 7.3]{PT} via a different method. 

\begin{corollary}
\label{loopofflagcomplexinP}
    Let $K$ be a flag complex on the vertex set $[m]$ and $A_1,\cdots,A_m$ be path connected $CW$-complexes such that $\Sigma A_i \in \mathcal{W}$ for all $i$. Then $\Omega \caa^K \in \mathcal{P}$.
    \qedno
\end{corollary}

The second important case of Theorem ~\ref{loopofkskelinP} is when $k=1$.

\begin{corollary}
\label{loopofgraphinP}
    Let $K$ be a graph on the vertex set $[m]$ and $A_1,\cdots,A_m$ be path connected $CW$-complexes such that $\Sigma A_i \in \mathcal{W}$ for all $i$. Then $\Omega \caa^K \in \mathcal{P}$.
    \qedno
\end{corollary}

Loop spaces of moment-angle complexes associated to graphs in certain cases have been studied. In particular, explicit decompositions of the loops of moment-angle complexes associated to wheel graphs and certain generalisations of wheel graphs \cite{T2}, and certain classes of generalised book graph \cite{St} have been given. This paper establishes that decompositions of this form exist for all graphs. While in principle an explicit decomposition could be obtained, in practice it would be difficult to do so.

Letting $A_i = S^{n-1}$ with $n \geq 2$ for all $i$ in Theorem ~\ref{loopofkskelinP} has consequences for generalised moment-angle complexes and moment-angle complexes. 

\begin{corollary}
\label{corgenMACofgraphinP}
    Let $k \geq 0$, and let $K$ be the $k$-skeleton of a flag complex. Then $\Omega (D^n,S^{n-1})^{K} \in \mathcal{P}$ where $n \geq 2$.
    \qedno
\end{corollary}

\begin{corollary}
\label{corMACofgraphinP}
    Let $k \geq 0$, and let $K$ be the $k$-skeleton of a flag complex. Then $\Omega \mathcal{Z}_{K} \in \mathcal{P}$.
    \qedno
\end{corollary}

It is interesting to note when the decomposition in Corollary ~\ref{corMACofgraphinP} arises from the fact that $\mathcal{Z}_K \in \mathcal{W}$. In the case of $K^1$, it is shown in \cite[Theorem 11.8]{IK2} that $K^1$ has $\mathcal{Z}_{K^1} \in \mathcal{W}$ if and only if $K^1$ is chordal. In particular, if $K^1$ is not chordal, then $\mathcal{Z}_{K^1}$ is not in $\mathcal{W}$. However, Corollary ~\ref{corMACofgraphinP} implies that nevertheless, $\Omega \mathcal{Z}_{K^1}$ is still in $\mathcal{P}$. In the case of $K$ itself, a similar result is true \cite[Theorem 6.4]{PT}, namely that $\mathcal{Z}_K \in \mathcal{W}$ iff $K^1$ is chordal.

To prove Theorem ~\ref{loopofkskelinP}, we will show that $\mathcal{P}$ is closed under retracts. This result was stated in \cite[p. 224]{PT} without proof, so a proof is provided here. The main tool that is used in the proof of this is the atomicity of loops on spheres when localised at certain primes (see Theorem ~\ref{atomicityoflocalisedloopsphere}). Let $K$ be a simplicial complex with a decomposition as $K = K_1 \cup_L K_2$ where $L$ is a full subcomplex of both $K_1$ and $K_2$. Closedness of $\mathcal{P}$ under retractions is applied to show that if $\Omega \caa^{K_1} \in \mathcal{P}$ and $\Omega \caa^{K_2} \in \mathcal{P}$, then $\Omega \caa^K \in \mathcal{P}$. This then allows us to prove the main result by an inductive argument.

The decomposition in Theorem ~\ref{loopofkskelinP} fits into a wider story related to loop space decompositions of spaces. Localise at a prime $p$. Given a space $X$, one may wish to find a decomposition of $\Omega X$ into a product of spaces, where each space in the product is indecomposable. Spheres $S^n$ where $n \in \{1,3,7\}$ and loops on simply connected spheres $\Omega S^{2m+1}$, where $m \geq 1$ are examples of indecomposable $H$-spaces. In a series of papers \cite{CMN1,CMN2,CMN3}, Cohen, Moore and Neisendorfer defined spaces $S^{2m+1}\{p^r\}$ and $T^{2m+1}\{p^r\}$ for $r \geq 1$ and $m \geq 1$, for which the loop space of a Moore space decomposes as a finite type product of these spaces. The spaces $S^{2m+1}\{p^r\}$ are indecomposable, and the spaces $T^{2m+1}\{p^r\}$ are indecomposable except for $T^3\{p\}$, in which case there is a homotopy equivalence $T^3\{p\} \simeq T^{2p+1}\{p\} \times U^1$, where $U^1$ is some indecomposable space. Anick \cite{An} conjectured that if $X$ is a finite, connected $CW$-complex, then localised at almost all primes $p$, $\Omega X$ decomposes as a finite type product of indecomposable spaces consisting of spheres, loops on simply connected spheres, $S^{2m+1}\{p^r\}$, $T^{2m+1}\{p^r\}$ and $U^1$. Theorem ~\ref{loopofkskelinP} verifies Anick's conjecture for polyhedral products $\caa^{K}$ where $\Sigma A_i \in \mathcal{W}$ and $K$ is the $k$-skeleton of a flag complex, and does so without the need to localise.

In Section ~\ref{sec:prelim}, some preliminary results in linear algebra and homotopy theory that will be required are introduced. In Section ~\ref{sec:retract}, we prove that the retract of a space in $\mathcal{P}$ is in $\mathcal{P}$. In Section ~\ref{sec:ApptoPP}, this is applied to polyhedral products to prove Theorem ~\ref{loopofkskelinP}.

\subsection*{Acknowledgements} The author would like to thank Stephen Theriault for the many helpful discussions during the preparation of this work. The author would also like to thank the referee for many valuable comments which have improved the paper, and also for spotting a gap in the original submission. 

\section{Preliminary Material}
\label{sec:prelim}

\subsection{Idempotent Matrices}
\label{sec:idempmat} 
In this section, we state and prove the basic properties of idempotent matrices that will be required in Section ~\ref{sec:retract}. Denote by $M_{n}(\mathbb{Z})$ the set of $n$ x $n$ matrices with integer entries. A matrix $A \in M_n(\mathbb{Z})$ is \textit{idempotent} if $A^2 = A$. Let $N(A)$ and $C(A)$ denote the null space and column space of $A$ respectively. Recall that the null space and column space of a matrix is the kernel and image of the corresponding linear map. The following result gives a decomposition of $\mathbb{Z}^n$ in terms of the null space and column space of an idempotent matrix. This Lemma is given as an exercise in \cite[p. 163]{L}, so we provide a proof here.

\begin{lemma}
\label{Zndecompidemp}
Let $A \in M_n(\mathbb{Z})$ be an idempotent matrix. Then $\mathbb{Z}^n \cong N(A) \oplus C(A)$. 
\end{lemma}
\begin{proof}
If $A$ is the zero matrix, then $C(A) = \{0\}$ and $\mathbb{Z}^n = N(A)$, and so the result holds in this case. Now suppose $A$ is non-trivial. First, we show that $\mathbb{Z}^n = N(A) + C(A)$. Clearly, $N(A)$ and $C(A)$ are subspaces of $\mathbb{Z}^n$, and so $N(A) + C(A) \subseteq \mathbb{Z}^n$. For the opposite inclusion, let $x \in \mathbb{Z}^n$. Write $x$ as $x = Ax-(Ax-x)$. Applying $A$ to $Ax-x$ and using the fact that $A$ is idempotent, we obtain \[A(Ax-x) = A^2x-Ax = Ax-Ax = 0.\] Therefore, since $Ax \in C(A)$ and $Ax-x \in N(A)$, $\mathbb{Z}^n \subseteq N(A) + C(A)$. Hence, $\mathbb{Z}^n = N(A) + C(A)$.

Now we show that $N(A) \cap C(A) = \{0\}$. The zero vector is contained in $N(A)$ and $C(A)$. Let $x \in N(A) \cap C(A)$. Since $x \in C(A)$, there exists $x' \in \mathbb{Z}^n$ such that $Ax' = x$. Applying $A$ to $x$ and using the fact that $x \in N(A)$, we obtain \[0 = Ax = A^2x' = Ax' = x.\] Therefore, $N(A) \cap C(A) = \{0\}$ and so $\mathbb{Z}^n = N(A) \oplus C(A)$. 
\end{proof} 

The next result describes how an idempotent matrix acts on an element of the column space. The proof is immediate from the definition of an idempotent matrix.

\begin{lemma}
\label{idemponcolumn}
    Let $A \in M_n(\mathbb{Z})$ be an idempotent matrix and let $x \in C(A)$. Then $Ax = x$. 
    \qedno
\end{lemma}

The final result describes the properties of the components of a vector $v \in \mathbb{Z}^n$ which extends to a basis of $\mathbb{Z}^n$. The result is clear from the contrapositive.

\begin{lemma}
\label{vectextendbasis}
Let $v = (v_1,\cdots,v_n)^T \in \mathbb{Z}^n$ be a vector which extends to a basis of $\mathbb{Z}^n$. Then the greatest common divisor of the non-zero components $v_1,\cdots,v_n$ is $1$. Moreover, one of $v_1,\cdots,v_n$ is odd.
\qedno
\end{lemma}

\subsection{Atomicity of loops on spheres} In this section, we recall the notion of atomic spaces. A simply connected topological space $X$ is \textit{atomic} \cite[Section 4]{CMN3} if any self map $f:X \rightarrow X$ inducing an isomorphism in the lowest non-vanishing degree in homology is a homotopy equivalence. A space $X$ is \textit{decomposable} if it is homotopy equivalent to a product $A \times B$ where $A$ and $B$ are not contractible. A space is \textit{indecomposable} if it is not decomposable. The study of atomic spaces is useful since atomic spaces are indecomposable. In Section ~\ref{sec:retract}, we will be interested in the atomicity properties of $\Omega S^n$. In particular, the following result is from \cite[Corollary 5.2]{CPS}.

\begin{theorem}
\label{atomicityoflocalisedloopsphere}
Let $p$ be a prime, and let $f$ be a self-map of $\Omega^k S^{m+1}$, $k<m$, which induces an isomorphism on the (least non-vanishing) homology group $H_{m+1-k}(\Omega^k S^{m+1}; \mathbb{Z}/p\mathbb{Z})$. If $p>2$, we suppose that $m$ is even, and if $p = 2$, we suppose $m \notin \{1,3,7\}$. Then $f$ is a $p$-local homotopy equivalence.
\qedno
\end{theorem}

Theorem ~\ref{atomicityoflocalisedloopsphere} implies that localised at any prime $p$, $\Omega S^{n}$ is atomic for $n$ odd, and when $n$ is even and $n \notin \{1,3,7\}$, $\Omega S^n$ is only atomic when localised at 2. The following result of Serre \cite{Se} shows that, localised at an odd prime, the loop space of an even dimensional sphere is decomposable. 

\begin{theorem}
\label{oddprimessplitting}
    Let $p$ be an odd prime. There is a $p$-local homotopy equivalence \[\Omega S^{2n} \simeq S^{2n-1} \times \Omega S^{4n-1}.\eqno\qed\]
\end{theorem}

\subsection{James-Hopf maps and Hopf invariants}

In this section, we introduce the James-Hopf maps and prove basic properties of their induced map on homology that will be required in Section ~\ref{sec:retract}. All homology groups will be assumed to have integer coefficients unless otherwise stated.

Let $X$ be a path-connected $CW$-complex such that $H_*(X)$ is torsion free. Let $E: X \rightarrow \Omega \Sigma X$ be the suspension map. The Bott-Samelson theorem implies that $H_*(\Omega \Sigma X) \cong T(\tilde{H}_*(X))$ where $T$ is the tensor algebra functor. Moreover, $E_*$ induces the inclusion of $\tilde{H}_*(X)$ into $T(\tilde{H}_*(X))$. Let $e: \Sigma \Omega \Sigma X \xrightarrow{e} \bigvee_{k\geq 1} \Sigma X^{\wedge k}$ be the James decomposition \cite{J}, where $X^{\wedge k}$ is the $k$-fold smash product of $X$ with itself. The James-Hopf map $h_k: \Omega \Sigma X \rightarrow \Omega (\Sigma X^{\wedge k})$ is the adjoint of the composite \[\overline{h}_k:\Sigma \Omega \Sigma X \xrightarrow{\simeq} \bigvee\limits_{k\geq 1} \Sigma X^{\wedge k}\rightarrow{} \Sigma X^{\wedge k}\] where the righthand map is the pinch map. The case that is applicable to Section ~\ref{sec:retract} is $X = S^{n-1}$ and $k=2$. In this case, $h_2$ is a map from $\Omega S^n$ to $\Omega S^{2n-1}$, and we can describe the induced map $(h_2)_*$ on homology in degree $2n-2$. Note that $H_{2n-2}(\Omega S^n) \cong H_{2n-2}(\Omega S^{2n-1}) \cong \mathbb{Z}$.

\begin{lemma}
\label{JamesHopfidentity}
    The map \[(h_2)_*:H_{2n-2}(\Omega S^n) \rightarrow H_{2n-2}(\Omega S^{2n-1})\] is an isomorphism. In particular, the generator $\delta \in H_{2n-2}(\Omega S^n)$ maps to $\tau \in H_{2n-2}(\Omega S^{2n-1})$, where $\tau$ is a generator.
\end{lemma}
\begin{proof}
   Consider the composite \[h'_2:\Sigma \Omega S^{n} \xrightarrow{\Sigma h_2} \Sigma \Omega S^{2n-1} \xrightarrow{ev} S^{2n-1}\] where $ev$ is the evaluation map. The map $h'_2$ is homotopic to $\overline{h}_2$ since it is the adjoint of $h_2$. By definition of $\overline{h}_2$ as the composite of a homotopy equivalence followed by the pinch map, $(\overline{h}_2)_*$ sends a generator $\sigma \delta \in H_{2n-1}(\Sigma \Omega S^n)$ to a generator $\tau' \in H_{2n-1}(S^{2n-1})$. Therefore, $(h'_2)_*$ also sends $\sigma \delta$ to $\tau'$. Since $H_{2n-1}(\Sigma \Omega S^n)$, $H_{2n-1}(\Sigma \Omega S^{2n-1})$ and $H_{2n-1}(S^{2n-1})$ are isomorphic to $\mathbb{Z}$, the only possibility is that $(\Sigma h_2)_*$ and $ev_*$ are isomorphisms in degree $2n-1$. Therefore, $(\Sigma h_2)_*(\sigma \delta) = \sigma\tau$ where $\tau$ is a generator of $H_{2n-2}(\Omega S^{2n-1})$. From the homology suspension isomorphism, we obtain that $(h_2)_*(\delta) = \tau$.
\end{proof}

Now consider the case where $X = S^{2n-1}$. For $m,k \geq 1$ and maps $f:S^{m} \rightarrow Z$ and $g:S^{k} \rightarrow Z$, denote the Whitehead product of $f$ and $g$ by $[f,g]:S^{m+k+1} \rightarrow Z$, and denote its adjoint, the Samelson product, by $\langle \tilde{f},\tilde{g}\rangle: S^{m+k} \rightarrow \Omega Z$, where $\tilde{f}$ and $\tilde{g}$ are the adjoints of $f$ and $g$ respectively. In particular, let $id:S^{2n} \rightarrow S^{2n}$ be the identity map. 

\begin{lemma}
\label{loopedhopfmapmulti}
    The map \[(\Omega [id,id])_*:H_{4n-2}(\Omega S^{4n-1}) \rightarrow H_{4n-2}(\Omega S^{2n})\] sends a generator $\tau \in H_{4n-2}(\Omega S^{4n-1})$ to $2\delta \in H_{4n-2}(\Omega S^{2n})$ where $\delta$ is a generator of $H_{4n-2}(\Omega S^{2n})$.
\end{lemma}

\begin{proof}
 Consider the diagram \[\begin{tikzcd}
	{S^{4n-2}} \\
	{\Omega S^{4n-1}} && {\Omega S^{2n}}
	\arrow["E", from=1-1, to=2-1]
	\arrow["{\Omega[id,id]}", from=2-1, to=2-3]
	\arrow["{\left\langle E,E \right\rangle}", from=1-1, to=2-3]
\end{tikzcd}\] where $E:S^{2n-1} \rightarrow \Omega S^{2n}$ is the suspension map. The diagram homotopy commutes since $\left\langle E,E\right\rangle$ is the adjoint of $[id,id]$. Since $E$ induces the inclusion of the generator $\tau \in H_{4n-2}(\Omega S^{4n-1})$, its image under $(\Omega [id,id])_*$ is determined by its image under $\left\langle E,E\right\rangle$. The Samelson product commutes with homology in the sense that $\left\langle E,E\right\rangle_* = \left\langle E_*,E_*\right\rangle$ where the bracket on the right is the commutator in $H_*(\Omega S^{2n}) \cong T(\gamma)$. The map $E$ induces the inclusion of the generator $\gamma \in H_{2n-1}(\Omega S^{2n})$, and so by definition of the commutator, $\left\langle E_*(\gamma),E_*(\gamma)\right\rangle = 2\delta$, where $\delta$ is a generator of $H_{4n-2}(\Omega S^{2n})$.
\end{proof}

\subsection{Hurewicz images}

Let $X$ be a space. An element $x \in H_n(X)$ is said to be in the Hurewicz image if it is in the image of the Hurewicz homomorphism. We will require the following result about the Hurewicz image of even dimensional spheres in a certain degree.

\begin{lemma}
\label{Hurewiczevensphere}
    The Hurewicz image $\pi_{4n-2}(\Omega S^{2n}) \rightarrow H_{4n-2}(\Omega S^{2n})$ when $n \notin \{1,2,4\}$ is $2\mathbb{Z}$.
\end{lemma}
\begin{proof}
    Suppose that $f: S^{4n-2} \rightarrow \Omega S^{2n}$ is a map with odd Hurewicz image. By the universal property of the James construction, there exists an $H$-map $\overline{f}:\Omega S^{4n-1} \rightarrow \Omega S^{2n}$ such that \[\begin{tikzcd}
	{S^{4n-2}} \\
	{\Omega S^{4n-1}} & {\Omega S^{2n}}
	\arrow["E", from=1-1, to=2-1]
	\arrow["f", from=1-1, to=2-2]
	\arrow["{\overline{f}}", from=2-1, to=2-2]
\end{tikzcd}\] homotopy commutes, where $E$ is the suspension map. Let $\tau$ be a generator of $H_{4n-2}(\Omega S^{4n-1})$ and $\delta$ be a generator of $H_{4n-2}(\Omega S^{2n})$. By commutativity of the diagram and the fact that $f$ has odd Hurewicz image, $\overline{f}$ sends $\tau$ to $(2k+1)\delta$ for some $k$. Consider the composite \[\phi:\Omega S^{4n-1} \xrightarrow{\Delta} \Omega S^{4n-1} \times \Omega S^{4n-1} \xrightarrow{p_{-k} \times \overline{f}} \Omega S^{4n-1} \times \Omega S^{2n} \xrightarrow{\Omega[id,id] \times id} \Omega S^{2n} \times \Omega S^{2n} \xrightarrow{\mu} \Omega S^{2n},\] where $p_{-k}$ is the $-k$'th power map. By Lemma ~\ref{loopedhopfmapmulti} and definition of $\phi$, $\phi_*(\tau) = \delta$. Now consider the composite \[\psi:\Omega S^{4n-1} \xrightarrow{\phi} \Omega S^{2n} \xrightarrow{h_2} \Omega S^{4n-1}.\] By Lemma ~\ref{JamesHopfidentity} and definition of $\phi$, $\psi_*(\tau) = \tau$, and so $\psi$ induces an isomorphism on $H_{4n-2}(\Omega S^{4n-1})$. By Theorem ~\ref{atomicityoflocalisedloopsphere}, $\Omega S^{4n-1}$ is $2$-locally atomic. Therefore, we obtain that $\psi$ is a $2$-local homotopy equivalence, implying that $\Omega S^{4n-1}$ retracts off $\Omega S^{2n}$ when localised at $2$. However, by Theorem ~\ref{atomicityoflocalisedloopsphere}, $\Omega S^{2n}$ is $2$-locally atomic, and therefore indecomposable. Hence, $\Omega S^{2n}$ has no non-trivial retracts localised at the prime $2$. 
\end{proof}

\subsection{Preliminary loop space decompositions}

In this section, we state and prove some initial loop space decompositions which will be applied in Section ~\ref{sec:ApptoPP}. Let $K$ be a simplicial complex on $[m]$ and let $L$ be a full subcomplex of $K$ on $[n]$. It is well known (see for example \cite[Lemma 2.2.3]{DS}) that the projection map $\prod\limits_{i=1}^{m} X_i \rightarrow \prod\limits_{j=1}^n X_{j}$ restricts to a map $\uxa^{K} \rightarrow \uxa^{L}$, which is a right inverse for the map $\uxa^L \rightarrow \uxa^K$. Note that a full subcomplex $L'$ of $L$ is also a full subcomplex of $K$, and this fact will often be used without comment.

There are two main results which will be used in Section ~\ref{sec:ApptoPP}. The first result was proved in \cite[Theorem 7.2]{GT1}. If $X$ and $Y$ have basepoints $x_0$ and $y_0$ respectively,  the \textit{right half-smash} is defined by $X \rtimes Y = X \times Y/(* \times Y)$ and the \textit{left half-smash} is defined by $X \ltimes Y = X \times Y/(X \times *)$. The \textit{reduced join} is defined by $X * Y = (X \times I \times Y)/\sim$, where $I$ is the unit interval, $(x,0,y) \sim (x,0,y')$, $(x,1,y) \sim (x',1,y)$ and $(x_0,t,y_0) \sim (x_0,0,y_0)$ for all $x,x' \in X$, $y,y' \in Y$ and $t \in I$.

\begin{proposition}
\label{emptysetdecomposition}
    Let $K_1$ be a simplicial complex on the vertex set $\{1,\cdots m\}$, $K_2$ a simplicial complex on the vertex set $\{\ell+1,\cdots,n\}$, and $\tau$ be a common face of $K_1$ and $K_2$ on the vertex set $\{\ell+1,\cdots,m\}$, where $\ell < m < n$. Then there is a homotopy equivalence \[\caa^{K_1 \cup_\tau K_2} \simeq (\mathcal{A}*\mathcal{A}') \vee (\caa^{K_1} \rtimes \mathcal{A}') \vee (\mathcal{A} \ltimes \caa^{K_2})\] where $\mathcal{A} = \prod_{i=1}^{\ell} A_i$ and $\mathcal{A}' = \prod_{i=m+1}^{n} A_i$.
    \qedno
\end{proposition}

The next main result is from \cite[Theorem 1.1]{T1}.

\begin{proposition}
\label{fullsubcomplexdecomp}
Let $K_1$ be a simplicial complex on the vertex set $\{1,\cdots,m\}$, $K_2$ a simplicial complex on the vertex set $\{\ell+1,\cdots,n\}$, and $L$ a full subcomplex of both $K_1$ and $K_2$ on the vertex set $\{\ell+1,\cdots,m\}$, where $\ell < m < n$. Then there is a homotopy fibration \[(\mathcal{A}*\mathcal{A'}) \vee (G \rtimes \mathcal{A}') \vee (\mathcal{A} \ltimes H) \rightarrow \caa^{K_1\cup_L K_2} \rightarrow \caa^L\] where $\mathcal{A} = \prod_{i=1}^\ell A_i$, $\mathcal{A}' = \prod_{j=m+1}^n A_i$, and $G$ and $H$ are the homotopy fibres of the retractions $\caa^{K_1} \rightarrow \caa^L$ and $\caa^{K_2} \rightarrow \caa^L$ respectively. Further, this fibration splits after looping to give a homotopy equivalence \[\Omega\caa^{K_1 \cup_L K_2} \simeq \Omega \caa^L \times \Omega((\mathcal{A} * \mathcal{A}') \vee (G \rtimes \mathcal{A}') \vee (\mathcal{A} \ltimes H)).\eqno\qed\]
\end{proposition}

\begin{remark}
The loop of the decomposition in Proposition ~\ref{emptysetdecomposition} can be obtained from Proposition ~\ref{fullsubcomplexdecomp}. However, the proof of Proposition ~\ref{fullsubcomplexdecomp} requires that $M$ is non-empty, whereas in Proposition ~\ref{emptysetdecomposition}, $\tau$ can be the empty set.
\end{remark}

The aim of Section ~\ref{sec:ApptoPP} is to use the decompositions in Proposition ~\ref{emptysetdecomposition} and Proposition ~\ref{fullsubcomplexdecomp} to show that the property of loop spaces of polyhedral products being in $\mathcal{P}$ is closed under taking pushouts of simplicial complexes over a common full subcomplex. First, we give a decomposition of $\Omega(X \ltimes Y)$ for spaces $X$ and $Y$. Observe there is a projection map $X \ltimes Y \rightarrow Y$ given by projecting onto $Y$.

\begin{lemma}
\label{loophalfsmash}
Let $X$ and $Y$ be path-connected, $CW$-complexes. Then there exists a homotopy fibration \[X * \Omega Y \rightarrow X \ltimes Y \rightarrow Y.\] Furthermore, this splits after looping to give a homotopy equivalence \[\Omega(X \ltimes Y) \simeq \Omega (X * \Omega Y) \times \Omega Y.\]
\end{lemma}

\begin{proof}
Consider the commutative diagram \[\begin{tikzcd}
	{\Omega Y} & {X \times \Omega Y} & X \\
	{*} & X & {X \times Y} \\
	Y & Y & Y
	\arrow[Rightarrow, no head, from=3-1, to=3-2]
	\arrow[Rightarrow, no head, from=3-3, to=3-2]
	\arrow["{i_X}", from=2-2, to=2-3]
	\arrow[from=2-2, to=2-1]
	\arrow[from=2-1, to=3-1]
	\arrow["{*}", from=2-2, to=3-2]
	\arrow["{\pi_Y}", from=2-3, to=3-3]
	\arrow["{\pi_X}", from=1-2, to=2-2]
	\arrow["{\pi_X}", from=1-2, to=1-3]
	\arrow["{i_X}", from=1-3, to=2-3]
	\arrow["{\pi_{\Omega Y}}"', from=1-2, to=1-1]
	\arrow[from=1-1, to=2-1]
\end{tikzcd}\] where the columns are homotopy fibrations, $i_X$ is the inclusion and $\pi_X$, $\pi_Y$ and $\pi_{\Omega Y}$ are the projections onto $X$, $Y$ and $\Omega Y$ respectively. Observe that the homotopy pushout of the top row is $X * \Omega Y$, the homotopy pushout of the middle row is $X \ltimes Y$ and the induced map from $X \ltimes Y$ to $Y$ is the projection map. Therefore by \cite[p.180]{F}, there is a homotopy fibration \[X * \Omega Y \rightarrow X \ltimes Y \rightarrow Y.\] Moreover, the projection $X \ltimes Y \rightarrow Y$ has a right homotopy inverse given by the inclusion map $Y \hookrightarrow X \ltimes Y$, which implies that there is a homotopy equivalence \[\Omega(X \ltimes Y) \simeq \Omega (X * \Omega Y) \times \Omega Y.\qedhere\]
\end{proof} 

Before determining conditions on $X$ and $Y$ for $\Omega (X \ltimes Y)$ to be in $\mathcal{P}$, we prove some relations between spaces in $\mathcal{W}$ and spaces in $\mathcal{P}$.

\begin{lemma}
\label{joinofWandPinW}
    Let $X$ be a space such that $\Sigma X \in \mathcal{W}$ and let $A_1,\cdots,A_m$ be spaces in $\mathcal{P}$, then \[\Sigma (X \wedge A_1 \wedge \cdots \wedge A_m) \in \mathcal{W}.\]
\end{lemma}
\begin{proof}
 We proceed by induction. First consider the case $m = 1$. Since $A_1 \in \mathcal{P}$, $\Sigma A_1 \in \mathcal{W}$. There is a homeomorphism $\Sigma (X \wedge A_1) \cong X \wedge \Sigma A_1$. Therefore, distributing the wedge sum over the smash product implies $\Sigma (X \wedge A_1) \in \mathcal{W}$.

    Now suppose the result is true for $1 \leq m \leq k-1$ and consider the case $m = k$. There are  homeomorphisms \[\Sigma (X \wedge A_1 \wedge \cdots \wedge A_m) \cong \Sigma (A_1 \wedge X \wedge A_2 \wedge \cdots \wedge A_m) \cong A_1 \wedge \Sigma(X \wedge A_2 \wedge \cdots \wedge A_m).\] The inductive hypothesis implies $\Sigma (X \wedge A_2 \wedge  \cdots \wedge A_m) \in \mathcal{W}$. Therefore, $\Sigma (X \wedge A_2 \wedge  \cdots \wedge A_m) \simeq \Sigma W$ where $W$ is a wedge of spheres. Hence, there is a homotopy equivalence \[\Sigma (X \wedge A_1 \wedge \cdots \wedge A_m) \simeq X \wedge \Sigma W.\]  Since $\Sigma X \in \mathcal{W}$ by assumption, shifting the suspension coordinate and distributing the smash product over the wedge sum implies $X \wedge \Sigma W \in \mathcal{W}$.
\end{proof}

\begin{lemma}
\label{loophalfsmashinP}
Let $X$ and $Y$ be path-connected $CW$-complexes such that $\Sigma X \in \mathcal{W}$ and $\Omega Y \in \mathcal{P}$. Then \[\Omega (X \ltimes Y) \in \mathcal{P}.\]
\end{lemma}\begin{proof}
 By Lemma ~\ref{loophalfsmash}, $\Omega (X \ltimes Y) \simeq \Omega(X *\Omega Y) \times \Omega Y$. Since $X *\Omega Y \simeq \Sigma (X \wedge  \Omega Y)$, $\Sigma X \in \mathcal{W}$ and $\Omega Y \in \mathcal{P}$, Lemma ~\ref{joinofWandPinW} implies $\Sigma (X \wedge \Omega Y) \in \mathcal{W}$. Therefore the Hilton-Milnor theorem \cite{M} implies $\Omega (\Sigma (X \wedge \Omega Y)) \in \mathcal{P}$, and so $\Omega (X \ltimes Y) \in \mathcal{P}$.
\end{proof}

Now we state a result of Porter \cite[Theorem 1]{P} which gives a loop space decomposition of a wedge of spaces. Let $X$ be a pointed space. Denote by $X^{\vee k}$ the $k$-fold wedge sum of $X$ with itself.

\begin{lemma}
\label{incwedgeintoprod}
Let $X_1, \cdots X_m$ be path-connected $CW$-complexes. Then there exists a homotopy fibration \[\bigvee\limits_{k=2}^m \bigvee\limits_{1 \leq i_1 < \cdots < i_k \leq m} (\Sigma \Omega X_{i_1} \wedge \cdots \wedge \Omega X_{i_k})^{\vee(k-1)} \rightarrow \bigvee\limits_{i=1}^m X_i \hookrightarrow \prod\limits_{i=1}^m X_i.\] Moreover, this splits after looping.
\qedno
\end{lemma}
Lemma ~\ref{incwedgeintoprod} can be applied to show that if there are spaces $X_i$ such that $\Omega X_i \in \mathcal{P}$, then the loop space of the wedge of the $X_i$'s is in $\mathcal{P}$.

\begin{corollary}
\label{loopofwedgeofP}
Let $X_1 \cdots, X_n$ be spaces such that $\Omega X_i \in \mathcal{P}$. Then $\Omega \left(\bigvee_{i=1}^n X_i\right) \in \mathcal{P}$.
\end{corollary}
\begin{proof}
By Lemma ~\ref{incwedgeintoprod}, there is a homotopy equivalence \[\Omega \left(\bigvee\limits_{i=1}^m X_i\right) \simeq \prod\limits_{i=1}^m \Omega X_i \times \Omega \left(\bigvee\limits_{k=2}^m \bigvee\limits_{1 \leq i_1 < \cdots < i_k \leq m} (\Sigma \Omega X_{i_1} \wedge \cdots \wedge \Omega X_{i_k})^{\vee(k-1)}\right).\] The product $\prod_{i=1}^m \Omega X_i$ is in $\mathcal{P}$ since $\mathcal{P}$ is closed under products, so consider the complimentary factor. Since $\Omega X_{i_1}$, $\Sigma \Omega X_{i_1} \in \mathcal{W}$. Lemma ~\ref{joinofWandPinW} then implies \[\Sigma \Omega X_{i_1} \wedge \cdots \wedge \Omega X_{i_k} \in \mathcal{W},\] and so \[\bigvee\limits_{k=2}^m \bigvee\limits_{1 \leq i_1 < \cdots < i_k \leq m} (\Sigma \Omega X_{i_1} \wedge \cdots \wedge \Omega X_{i_k})^{\vee(k-1)} \in \mathcal{W}.\] Therefore, by the Hilton-Milnor Theorem \[\Omega \left(\bigvee\limits_{k=2}^m \bigvee\limits_{1 \leq i_1 < \cdots < i_k \leq m} (\Sigma \Omega X_{i_1} \wedge \cdots \wedge \Omega X_{i_k})^{\vee(k-1)}\right) \in \mathcal{P}.\]
\end{proof}

\section{Closure of $\mathcal{P}$ under retracts}
\label{sec:retract}

\subsection{Setup}
\label{setup}

In this section, homology will be assumed to have integer coefficients unless otherwise stated. Let $X \in \mathcal{P}$, and suppose there is a space $A$ which retracts off $X$, that is, there exist maps $f: A \rightarrow X$ and $g: X \rightarrow A$ such that the diagram \[\begin{tikzcd}
	A & X \\
	& A
	\arrow["f", from=1-1, to=1-2]
	\arrow["g", from=1-2, to=2-2]
	\arrow[Rightarrow, no head, from=1-1, to=2-2]
\end{tikzcd}\] homotopy commutes. In this section, we will show that $A$ is homotopy equivalent to a subproduct of $X$.

The product decomposition of $X$ implies there is a coalgebra isomorphism of $H_*(X)$ as a tensor product of exterior algebras corresponding to the spheres, and single-variable polynomial rings corresponding to the loops on spheres. Consider the composite $\phi:X \xrightarrow{g} A \xrightarrow{f} X$. Observe that $f \circ g \circ f \circ g \simeq f \circ g$ which implies that the induced map $\phi_*$ is an idempotent. To show that $A$ is homotopy equivalent to a subproduct of $X$, we will proceed in three stages. First, we consider the case where $H_*(A)$ contains a primitive generator in degree $m$ for $m \in \{1,2,3,6,7,14,4m\:|\: m \geq 1\}$, then we will consider the case where $H_*(A)$ contains a primitive generator in the Hurewicz image in degree $4m+2$, where $m \geq 2$, $m \neq 3$, and we will conclude by considering the case where $H_*(A)$ contains a primitive generator in degree $m$ where $m$ is odd and $m \notin \{1,3,7\}$. 

Each case requires an adaptation of the same core idea, and the notation defined in each subsection will be reused to reflect where the argument is the same. There is some notation that will be universal which we now define. Let $Y = S^n$ for $n \in \{1,3,7\}$ or $Y = \Omega S^{m}$ for $m \notin \{2,4,8\}$. Let $m_Y$ be the number of instances of $Y$ in the product decomposition of $X$. In particular, write $X$ as \[X \simeq \prod\limits_{i=1}^{m_Y} Y_i \times \prod\limits_{\alpha' \in \mathcal{I}'} Z_{\alpha'}\] where each $Y_i$ is an instance of $Y$ in the product decomposition of $X$, and each $Z_{\alpha'}$ are the spheres and loops on spheres that are not equal to $Y$. Denote by $H$ the lowest non-vanishing homology group of $Y$, and $H_i$ the lowest non-vanishing homology group of $Y_i$. Note that $H \cong H_i \cong \mathbb{Z}$ for all $i$. Let $\gamma_i$ be the primitive generator of $H_*(X)$ which is the image of a generator $\gamma_i' \in H_i$ under the map induced by the inclusion $Y_i \hookrightarrow X$. We will define maps $\rho_v:Y \rightarrow X$ and $\rho'_v:X \rightarrow Y$ such that the composite $Y \xrightarrow{\rho_v} X \xrightarrow{\phi} X \xrightarrow{\rho'_v} Y$ is a homotopy equivalence. Since $\phi$ factors through $A$ and $A$ is a $H$-space, this will imply that we obtain a homotopy equivalence $A \simeq Y \times A'$ for some space $A'$. An iterative approach will then be used to conclude that $A \in \mathcal{P}$.

\subsection{Case 1}
\label{spheresandonemod4}

In this subsection, we implicitly fix $Y$ to be either $Y = S^n$ for $n \in \{1,3,7\}$, $Y = \Omega S^{4m+1}$ for $m \geq 2$, or $Y = \Omega S^{4m+3}$ for $m \in \{0,1,3\}$. We show that if the homology of $A$ contains a primitive generator in the same degree as $H$, then $Y$ retracts off $A$.  

Observe that in this case, the set $\{\gamma_1,\cdots,\gamma_{m_Y}\}$ forms a basis of primitives in $H_n(X)$ if $Y = S^n$, $H_{4m}(X)$ if $Y = \Omega S^{4m+1}$, or $H_{4m+2}(X)$ if $Y = \Omega S^{4m+3}$. Since $\phi_*$ is a graded coalgebra map, it maps primitive elements to primitive elements of the same degree, and so $\phi_*(\gamma_i) = \sum_{j=1}^{m_Y} z_{i,j} \gamma_j$, where $z_{i,j} \in \mathbb{Z}$ for all $j$. Let $B_Y \in M_{m_Y}(\mathbb{Z})$ be the matrix with entries $z_{i,j}$. Since $\phi_*$ is an idempotent map, $B_Y$ is an idempotent matrix. 

Suppose $B_Y$ is not the zero matrix. Since $B_Y$ is idempotent, Lemma ~\ref{Zndecompidemp} implies that there exists an element $v = (y_1,\cdots,y_{m_Y})^T \in C(B_Y)$ which extends to a basis of $\mathbb{Z}^{m_Y}$. Therefore, by Lemma ~\ref{vectextendbasis}, the greatest common divisor of the non-zero components $y_1,\cdots,y_{m_Y}$ is $1$. By B\'ezout's Lemma, for $1 \leq i \leq m_Y$, there exists $c_i \in \mathbb{Z}$ such that $\sum_{i=1}^{m_Y} c_iy_i = 1$. Since $v \in C(B_Y)$, Lemma ~\ref{idemponcolumn} implies $B_Yv = v$. Let the vector $v$ correspond to the element $\sum_{i=1}^{m_Y} y_i \gamma_i$ in $H_*(X)$.

Let $d_k: S^n \rightarrow S^n$ be the degree $k$ map, and let $p_k:\Omega S^n \rightarrow \Omega S^n$ be the $k^{th}$ power map. Note that $d_k$ and $p_k$ both induce multiplication by $k$ in $H$ (in this case of $p_k$, this follows from the Hurewicz theorem).  Let $\psi_{k}$ be $d_{k}$ if $Y$ is a sphere or $p_{k}$ if $Y$ is the loops on a sphere. Define a map $\rho_v:Y\rightarrow X$ as the composite \[\rho_v:Y \xrightarrow{\Delta} \prod\limits_{i=1}^{m_Y} Y_i \xrightarrow{\prod\limits_{i=1}^{m_Y} \psi_{y_i}} \prod\limits_{i=1}^{m_Y} Y_i \hookrightarrow X\] where $\Delta$ is the diagonal map, and the right map is the inclusion. Now define a map $\rho'_v: X \rightarrow Y$ as the composite \[\rho'_v:X \xrightarrow{\pi} \prod\limits_{i=1}^{m_Y} Y_i \xrightarrow{\prod\limits_{i=1}^{m_Y}\psi_{c_i}} \prod\limits_{i=1}^{m_Y} Y_i \xrightarrow{\mu} Y\] where $\pi$ is the projection, $\mu$ is some choice of $m_Y$-fold $H$-space multiplication on $Y$, and the $c_i$'s have the property that $\sum_{i=1}^{m_Y} c_i y_i = 1$.

\begin{lemma}
\label{isoonbottomdegree}
 Suppose that $B_Y$ is not the zero matrix. Then, the composite \[e:Y \xrightarrow{\rho_v} X \xrightarrow{\phi} X \xrightarrow{\rho'_v} Y\] induces an isomorphism on $H$.
\end{lemma}
\begin{proof}
By definition, $(\rho_v)_*$ sends the generator $\gamma \in H$ to the element $v$ in $H_*(X)$. Since $v$ is in the column space of $B_Y$, Lemma ~\ref{idemponcolumn} implies that $v$ is fixed by $\phi_*$. By definition of $\rho'_v$, $(\rho'_v)_*$ sends $v$ to the generator $\gamma \in H$. Therefore, $e_*$ is an isomorphism on $H$.
\end{proof}

Lemma ~\ref{isoonbottomdegree} allows us to conclude that $e$ is a homotopy equivalence.

\begin{lemma}
\label{spheresretractoffA}
    Let $Y$ be a sphere $S^m$ for $m \in \{1,3,7\}$, loops on a sphere of the form $\Omega S^{4m+1}$ for $m \geq 1$, or $\Omega S^{4m+3}$ for $m \in \{0,1,3\}$. Suppose that $H_*(A)$ contains a primitive generator in degree $m$ if $Y$ is a sphere, in degree $4m$ if $Y = \Omega S^{4m+1}$, or in degree $4m+2$ if $\Omega S^{4m+3}$. Then $Y$ retracts off $A$.
\end{lemma}
\begin{proof}
    Since $H_n(A)$ contains a primitive generator, by definition of $\phi$, the matrix $B_Y$ is non-zero. If $Y$ is a sphere, then $H$ is the only non-vanishing homology group of $Y$. As $e_*$ is an isomorphism on $H$ by Lemma ~\ref{isoonbottomdegree} and $Y$ is an $H$-space, $e$ is a homotopy equivalence. The map $\phi$ factors through $A$, and so $Y$ retracts off $A$.

    If $Y = \Omega S^{4m+1}$ or $\Omega S^{4m+3}$, then $e$ induces an isomorphism on $H_k(Y)$, where $k = 4m$ if $Y = \Omega S^{4m+1}$ and $k = 4m+2$ if $Y = \Omega S^{4m+3}$. This implies that $e$ induces an isomorphism on $H_{k}(Y;\mathbb{Z}/p\mathbb{Z})$ for all primes $p$ and rationally. Therefore when localised at $p$ or rationally, Theorem ~\ref{atomicityoflocalisedloopsphere} implies $e$ is a homotopy equivalence. Since $e$ is a homotopy equivalence localised at every prime and rationally, $e$ is an integral homotopy equivalence. Hence, $Y$ retracts off $A$. 
\end{proof}

From the previous lemma, we obtain the following result.

\begin{proposition}
\label{spherescase}
    Let $X \in \mathcal{P}$ and $A$ be a space which retracts off $X$. Suppose that $H_*(A)$ contains a primitive generator in degree $m$ where $m \in \{1,2,3,6,7,14,4m\:|\: m \geq 1\}$. Then there is a homotopy equivalence \[A \simeq Y \times A'\] where $Y = S^m$ if $m \in \{1,3,7\}$, or $Y= \Omega S^{m+1}$ otherwise. Moreover, $A'$ retracts off $X$, and $H_m(A')$ contains one fewer primitive generator than $H_m(A)$.
\end{proposition}

\begin{proof}
 Since there is a primitive generator of degree $m$, the matrix $B_Y$ is non-zero, and so Lemma ~\ref{spheresretractoffA} implies that $Y$ retracts off $A$. This implies there is a map $r:A \rightarrow Y$ which has a right homotopy inverse.  Let $F$ be the homotopy fibre of $g$ and consider the homotopy fibration diagram \[\begin{tikzcd}
	F & {X'} & {A'} \\
	F & X & A \\
	& Y & Y
	\arrow["g", from=2-2, to=2-3]
	\arrow["r", from=2-3, to=3-3]
	\arrow["{g \circ r}", from=2-2, to=3-2]
	\arrow[Rightarrow, no head, from=3-2, to=3-3]
	\arrow[from=2-1, to=2-2]
	\arrow[from=1-1, to=1-2]
	\arrow[from=1-2, to=1-3]
	\arrow[from=1-3, to=2-3]
	\arrow[from=1-2, to=2-2]
	\arrow[Rightarrow, no head, from=1-1, to=2-1]
\end{tikzcd}\] where $A'$ and $X'$ are the homotopy fibres of $r$ and $g \circ r$ respectively. Since $A$ retracts off $X$, it is an $H$-space. The right homotopy inverse for $r$ implies there is a homotopy equivalence $A \simeq Y \times A'$. Observe that $A'$ has the same homology as $A$ except with one less primitive generator in degree $m$. Moreover, since $g$ has a right homotopy inverse and $X$ is an $H$-space, there are homotopy equivalences $X \simeq A \times F \simeq Y \times A' \times F$. Hence, $A'$ retracts off $X$.
\end{proof}

\subsection{Case 2}
\label{threemod4}

In this subsection, fix $Y$ to be $\Omega S^{4n+3}$ for $n \geq 2$, $n \neq 3$.  We show that if the homology of $A$ contains a primitive generator in $H_{4n+2}(A)$ which is in the Hurewicz image, then $\Omega S^{4n+3}$ retracts off $A$. In this case, the set $\{\gamma_1,\cdots,\gamma_{m_Y}\}$ does not form a basis of primitives in $H_{4n+2}(X)$, since there may be $\Omega S^{2n+2}$ terms in the product decomposition for $X$. Let $\overline{Y} = \Omega S^{2n+2}$. Write $X$ as \[X \simeq \prod\limits_{i=1}^{m_Y} \Omega S^{4n+3}_i \times \prod\limits_{j=1}^{m_{\overline{Y}}} \Omega S^{2n+2}_j \times \prod\limits_{\alpha' \in \mathcal{I}'} Z_{\alpha'}\] where each $Z_{\alpha'}$ are the spheres and loops on spheres that are not equal to $\Omega S^{4n+3}$ or $\Omega S^{2n+2}$. Let $\overline{\gamma}_i$ be the primitive generator of $H_*(X)$ which is the image of a generator $\overline{\gamma}_i' \in H_{4n+2}(\Omega S^{2n+2})$ under the map induced by the inclusion $\Omega S^{2n+2}_i \hookrightarrow X$. The set $\{\gamma_1,\cdots,\gamma_{m_Y},\overline{\gamma}_1,\cdots,\overline{\gamma}_{m_{\overline{Y}}}\}$ forms a basis of primitives in $H_{4n+2}(X)$. 

Consider a primitive generator $a \in H_{4n+2}(A)$ such that $a$ is in the Hurewicz image. Observe that $f_*(a)$ is primitive, in the Hurewicz image, and since $f_*$ is injective, $f_*(a)$ is non-zero. By Lemma ~\ref{Hurewiczevensphere}, $f_*(a) = \sum_{i=1}^{m_Y} y_i \gamma_i + \sum_{j=1}^{m_{\overline{Y}}} 2\overline{y}_j \overline{\gamma}_j$. Let $v = (y_1,\cdots,y_{m_Y},2\overline{y}_1,\cdots,2\overline{y}_{m_{\overline{Y}}})$ correspond to the element $f_*(a)$. By definition of $\phi$, im$(\phi_*) = $ im$(f_*)$, and $f_*(a)$ is a generator of im$(\phi_*)$. Therefore, by Lemma ~\ref{vectextendbasis}, the greatest common divisor of the components of $v$ is $1$. By B\'ezout's Lemma, for $1 \leq i \leq m_Y$ and $1 \leq j \leq m_{\overline{Y}}$, there exists $c_i,\overline{c}_j \in \mathbb{Z}$ such that $\sum_{i=1}^{m_Y} c_iy_i + \sum_{j=1}^{m_{\overline{Y}}} 2\overline{c}_j\overline{y}_j = 1$. Since $v \in \text{im}(\phi)$, Lemma ~\ref{idemponcolumn} implies $\phi_*(f_*(a)) = f_*(a)$. 

Let $\lambda_{k}$ be the composite \[\lambda_k:\Omega S^{4n+3} \xrightarrow{\Omega [id,id]} \Omega S^{2n+2} \xrightarrow{p_{k}} \Omega S^{2n+2}.\] By Lemma ~\ref{loopedhopfmapmulti}, $\lambda_k$ maps $\gamma$ to $2k\overline{\gamma}$. Define a map $\rho_v:\Omega S^{4n+3}\rightarrow X$ as the composite \[\rho_v:\Omega S^{4n+3} \xrightarrow{\Delta} \prod\limits_{i=1}^{m_Y} \Omega S^{4n+3}_i \times \prod\limits_{j=1}^{m_{\overline{Y}}} \Omega S^{4n+3}_j \xrightarrow{\prod\limits_{i=1}^{m_Y}p_{y_i} \times \prod\limits_{j=1}^{m_{\overline{Y}}}\lambda_{\overline{y}_{j}}} \prod\limits_{i=1}^{m_Y} \Omega S^{4n+3}_i \times \prod\limits_{j=1}^{m_{\overline{Y}}} \Omega S^{2n+2}_j \hookrightarrow X\] where $\Delta$ is the diagonal map, and the right map is the inclusion. Now define a map $\rho'_v: X \rightarrow \Omega S^{4n+3}$ as the composite \[\rho'_v:X \xrightarrow{\pi} \prod\limits_{i=1}^{m_Y} \Omega S^{4n+3}_i \times \prod\limits_{j=1}^{m_{\overline{Y}}} \Omega S^{2n+2}_j \xrightarrow{\prod\limits_{i=1}^{m_Y}id \times \prod\limits_{j=1}^{m_{\overline{Y}}}h_2} \prod\limits_{i=1}^{m_Y} \Omega S^{4n+3}_i \times \prod\limits_{j=1}^{m_{\overline{Y}}} \Omega S^{4n+3}_j \]\[\xrightarrow{\prod\limits_{i=1}^{m_Y}p_{c_i} \times \prod\limits_{j=1}^{m_{\overline{Y}}}p_{\overline{c}_j}} \prod\limits_{i=1}^{m_Y} \Omega S^{4n+3}_i \times \prod\limits_{j=1}^{m_{\overline{Y}}} \Omega S^{4n+3}_j \xrightarrow{\mu} \Omega S^{4n+3}\] where $\pi$ is the projection, $\mu$ is some choice of $m_Y$-fold $H$-space multiplication on $\Omega S^{4n+3}$, and the $c_i$'s and $\overline{c}_j$'s have the property that $\sum_{i=1}^{m_Y} c_i y_i + \sum_{j=1}^{m_{\overline{Y}}} 2\overline{c}_j\overline{y}_j= 1$. By definition, $(\rho_v)_*$ sends the generator $\gamma \in H$ to the element $v$ in $H_*(X)$, and by definition of $\rho'_v$, $(\rho'_v)_*$ sends $v$ to the generator $\gamma \in H$. Therefore, arguing as in Lemma ~\ref{isoonbottomdegree}, Lemma ~\ref{spheresretractoffA} and Proposition ~\ref{spherescase}, we obtain the following.

\begin{proposition}
\label{loopspheresmod3case}
    Let $X \in \mathcal{P}$ and $A$ be a space which retracts off $X$. Suppose $H_{4n+2}(A)$, $n \geq 2$, $n \neq 3$, contains a primitive generator in the Hurewicz image. Then there is a homotopy equivalence \[A \simeq \Omega S^{4n+3} \times A'\] where $A'$ retracts off $X$, and $H_{4n+2}(A')$ contains one fewer primitive generator in the Hurewicz image than $H_{4n+2}(A)$.
    \qedno
\end{proposition}

\subsection{Case 3}
\label{evendimensionalspheres}

In this subsection, fix $Y$ to be $\Omega S^{2n}$ for $n \notin \{1,2,4\}$. We show that if the homology of $A$ contains a primitive generator in $H_{2n-1}(A)$, then $\Omega S^{2n}$ retracts off $A$. Observe that in this case, the set $\{\gamma_1,\cdots,\gamma_{m_Y}\}$ forms a basis of primitives in $H_{2n-1}(X)$. Therefore, if $H_{2n-1}(A)$ contains a primitive generator, arguing as in Subsection ~\ref{spheresandonemod4}, we obtain \begin{enumerate}
    \item a generator of im$(\phi_*)$, $\sum_{i=1}^{m_Y} y_i \gamma_i \in H_{2n-1}(X)$,
    \item a vector $v = (y_1,\cdots,y_{m_Y})$ corresponding to the generator in (1) such that the greatest common divisor of the non-zero components is $1$;
    \item a composite \[\rho_v:\Omega S^{2n} \xrightarrow{\Delta} \prod\limits_{i=1}^{m_Y} \Omega S^{2n}_i \xrightarrow{p_{y_i}} \prod\limits_{i=1}^{m_Y} \Omega S^{2n}_i \hookrightarrow X\] such that a generator $\gamma \in H_{2n-1}(\Omega S^{2n})$ maps to $v$;
    \item a composite \[\rho'_v:X \xrightarrow{\pi} \prod\limits_{i=1}^{m_Y} \Omega S^{2n}_i \xrightarrow{\prod\limits_{i=1}^{m_Y}p_{c_i}} \prod\limits_{i=1}^{m_Y} \Omega S^{2n}_i \xrightarrow{\mu} \Omega S^{2n}\] where $\pi$ is the projection, $\mu$ is some choice of $m_Y$-fold $H$-space multiplication on $\Omega S^{2n}$, the $c_i$'s have the property that $\sum_{i=1}^{m_Y} c_i y_i = 1$, and the map $(\rho'_v)_*$ maps $v$ to $\gamma$;
    \item the composite $e:(\rho'_v)_* \circ \phi_* \circ (\rho_v)_*$ is an isomorphism on $H_{2n-1}(\Omega S^{2n})$.
\end{enumerate} 

In Subsection ~\ref{spheresandonemod4}, this was enough to conclude that the corresponding loop on sphere retracted off of $A$. In this case, $e$ may not be a homotopy equivalence since $\Omega S^{2n}$ is not atomic at odd primes. However, we can adjust the maps $\rho_{v}$ and $\rho'_{v}$ to define a map $\overline{e}'$ which is a homotopy equivalence. In particular, these maps will agree with $\rho_v$ and $\rho'_v$ respectively in degree $2n-1$, but may differ in degree $4n-2$ so that $\overline{e}'$ is an isomorphism on both $H_{2n-1}(\Omega S^{2n})$ and $H_{4n-2}(\Omega S^{2n})$.

Since the generators of $H_*(\Omega S^{2n})$ in degree $2n-1$ and $4n-2$ are divisors of elements in the Hurewicz image, the map induced by the $k^{th}$ power map sends a generator $\gamma \in H_{2n-1}(\Omega S^{2n})$ to $k \gamma$, and a generator $\delta \in H_{4n-2}(\Omega S^{2n})$ to $k\delta$. Recall that $v = (y_1,\cdots,y_{m_Y})^T$, and $(\rho_v)_*$ maps a generator $\gamma \in H_{2n-1}(\Omega S^{2n})$ to $v$ in $H_*(X)$. Let $\delta_i$ be the primitive generator of $H_{4n-2}(X)$ which is the image of a generator $\delta_i' \in H_{4n-2}(\Omega S^{2n}_i)$ under the map induced by the inclusion $\Omega S^{2n}_i \hookrightarrow X$. By definition of $\rho_v$, for a suitable choice of generator $\delta \in H_{4n-2}(\Omega S^{2n})$, $(\rho_v)_*$ maps $\delta$ to the element $\sum_{i=1}^{m_Y} y_i \delta_i$. 

 Let $\overline{Y} = \Omega S^{4n-1}$. Write $X$ as \[X \simeq \prod\limits_{i=1}^{m_Y} \Omega S^{2n}_i \times \prod\limits_{j=1}^{m_{\overline{Y}}} \Omega S^{4n-1}_j \times \prod\limits_{\alpha' \in \mathcal{I}'} Z_{\alpha'}\] where each $Z_{\alpha'}$ are the spheres and loops on spheres that are not equal to $\Omega S^{2n}$ or $\Omega S^{4n-1}$. Let $\overline{\delta}_i$ be the primitive generator of $H_{4n-2}(X)$ which is the image of a generator $\overline{\delta}_i' \in H_{4n-2}(\Omega S^{4n-1}_i)$ under the map induced by the inclusion $\Omega S^{4n-1}_i \hookrightarrow X$. Observe that the set $\{\delta_1,\cdots,\delta_{m_Y},\overline{\delta}_1,\cdots,\overline{\delta}_{m_{\overline{Y}}}\}$ forms a basis of the primitives in $H_{4n-2}(X)$. Let $w = (y_1,\cdots,y_{m_{Y}},0,\cdots,0)^T$ be the vector defined by taking the coefficients of $\sum_{i=1}^{m_Y} y_i \delta_i$. Since $w$ is primitive, $\phi_*$ maps $w$ to an element of the form $\sum_{i=1}^{m_Y} y_i' \delta_i + \sum_{j=1}^{m_{\overline{Y}}}\overline{y}_i \overline{\delta}_j$. Let $w' = (y_1',\cdots,y'_{m_{Y}},\overline{y}_1,\cdots,\overline{y}_{m_{\overline{Y}}})^T$ be the vector containing the coefficients of $\sum_{i=1}^{m_Y} y_i' \delta_i + \sum_{j=1}^{m_{\overline{Y}}}\overline{y}_i \overline{\delta}_j$. The components of $w'$ can be related to the components of $v$. 

\begin{lemma}
\label{vectorssamemod2}
    The components $y_i$ are equal to $y_i'$ modulo $2$. Moreover, there exists $1 \leq j \leq m_Y$ such that $y_j$ is odd.
\end{lemma}
\begin{proof}
Since the vector $v$ extends to a basis of $\mathbb{Z}^{m_Y}$, by Lemma ~\ref{vectextendbasis}, at least one of $y_1,\cdots,y_m$ must be odd. Consider a component $y_i$ of $v$ which is odd, and consider the composite \[\psi:\Omega S^{2n} \xrightarrow{\rho_v} X \xrightarrow{\phi} X \xrightarrow{\pi_i} \Omega S^{2n}_i\] where $\pi_i$ is the projection map. By definition, $\psi_*$ sends the generator $\gamma \in H_{2n-1}(\Omega S^{2n})$ to $y_i\gamma$, and the generator $\delta \in H_{4n-2}(\Omega S^{2n})$ to $y_i'\delta$. Since $y_i$ is odd, $\psi_*$ is an isomorphism in degree $2n-1$, and so by Theorem ~\ref{atomicityoflocalisedloopsphere}, $\psi$ is a $2$-local homotopy equivalence. This implies that integrally, $y_i'$ must be odd and so $y_i' = y_i+2k_i$ for some $k_i \in \mathbb{Z}$, and this holds for all $i$ for which $y_i$ is odd. 

Now fix a component $y_i$ of $v$ which is odd, and consider a component $y_j$ which is even. Consider the composite \[\psi': \Omega S^{2n} \xrightarrow{\rho_v} X \xrightarrow{\phi} X \xrightarrow{\pi_{i,j}} \Omega S^{2n}_i \times \Omega S^{2n}_j \xrightarrow{\mu} \Omega S^{2n}\] where $\pi_{i,j}$ is the projection onto $\Omega S^{2n}_i \times \Omega S^{2n}_j$ and $\mu$ is the loop space multiplication. By definition, $(\psi')_*$ sends the generator $\gamma \in H_{2n-1}(\Omega S^{2n})$ to $(y_i + y_j)\gamma$ and the generator $\delta \in H_{4n-2}(\Omega S^{2n})$ to $(y_i' + y_j')\delta$. Since $y_i$ is odd and $y_j$ is even, $y_i+y_j$ is odd. Therefore, $(\psi')_*$ on $H_{2n-1}(\Omega S^{2n};\mathbb{Z}/2\mathbb{Z})$ is an isomorphism, and so by Theorem ~\ref{atomicityoflocalisedloopsphere}, $\psi'$ is a $2$-local homotopy equivalence. Therefore, $y_i'+y_j'$ must be odd. As $y_i$ is odd by assumption, $y_i'$ is odd by the previous paragraph, which implies that $y_j'$ must be even. This implies that $y_j' = y_j + 2k_j$ for some $k_j \in \mathbb{Z}$, and this holds for all $j$ for which $y_j$ is even. 
\end{proof}

Lemma \ref{vectorssamemod2} implies that $w' = (y_1+2k_1,\cdots, y_{m_{Y}}+2k_{m_{Y}},\overline{y}_1,\cdots,\overline{y}_{m_{\overline{Y}}})$ for $k_1,\cdots,k_{m_{Y}} \in \mathbb{Z}$.  The next result shows that there exists a vector $\overline{w} \in H_{4n-2}(X)$ with similar properties to $v$. For two vectors $u = (x_1,\cdots,x_n)^T$ and $u' = (x_1',\cdots,x_n')^T$, we say that $u$ and $u'$ are equal modulo $2$ if $x_i \cong x_i' \text{ mod } 2$ for all $i$. 

\begin{lemma}
\label{vectorsecondbot}
    There exists a vector $\overline{w}$ which is equal to $w'$ modulo $2$, and whose non-zero components have greatest common divisor $1$. Moreover, $\phi_*(\overline{w}) = \overline{w}$.
\end{lemma}

\begin{proof}
Recall that $y_i+2k_i = y_i'$. Let $d$ be the greatest common divisor of the non-zero components of $y_1+2k_1,\cdots, y_{m_{Y}}+2k_{m_{Y}},\overline{y}_1,\cdots,\overline{y}_{m_{\overline{Y}}}$. Then $w' = d\overline{w}$ for some vector $\overline{w}$. Lemma ~\ref{vectextendbasis} implies that one of $y_1,\cdots,y_{m_{Y}}$ is odd, and so it follows that $d$ is odd. By definition of $\overline{w}$, the greatest common divisor of the non-zero components of $\overline{w}$ is 1. Since $\phi_*$ is an idempotent map, \[d\overline{w} = w' = \phi_*(w') = d\phi_*(\overline{w}),\] which implies that $\phi_*(\overline{w}) = \overline{w}$. Moreover, since $d$ is odd, it follows that $\overline{w}$ is equal to $w'$ modulo $2$.
\end{proof}

 Lemma ~\ref{vectorsecondbot} implies there exists a vector $\overline{w} = (y_1+2k'_1,\cdots, y_{m_{Y}}+2k'_{m_Y},\overline{y}_1+2\overline{k}_1,\cdots,\overline{y}_{m_{\overline{Y}}}+2\overline{k}_{m_{\overline{Y}}})$ for $k'_1,\cdots,k'_{m_Y},\overline{k}_1,\cdots,\overline{k}_{m_{\overline{Y}}} \in \mathbb{Z}$ such that the greatest common divisor of \[y_1+2k'_1,\cdots, y_{m_{Y}}+2k'_{m_Y},\overline{y}_1+2\overline{k}_1,\cdots,\overline{y}_{m_{\overline{Y}}}+2\overline{k}_{m_{\overline{Y}}}\] is 1. For an integer $k$, let $\lambda_{k}:\Omega S^{4n-1} \rightarrow \Omega S^{2n}$ be $p_k \circ\Omega[id,id]$, where $p_k$ is the $k^{th}$ power map and $id:S^{2n} \rightarrow S^{2n}$ is the identity map. By Lemma ~\ref{loopedhopfmapmulti}, $(\Omega [id,id])_*$ sends a generator $\tau \in H_{4n-2}(\Omega S^{4n-1})$ to the element $2\delta \in H_{4n-2}(\Omega S^{2n})$ where $\delta$ is a generator of $H_{4n-2}(\Omega S^{2n})$. Therefore by definition of $\lambda_{k}$, $(\lambda_{k})_*$ sends $\tau$ to $2k\delta$. Let $h_2:\Omega S^{2n} \rightarrow \Omega S^{4n-1}$ be the $2^{nd}$ James-Hopf invariant. By Lemma ~\ref{JamesHopfidentity}, $(h_2)_*$ sends $\delta$ to $\tau$. For an integer $k$, let $\eta_{k}$ be the composite \[\eta_{k}:\Omega S^{2n}\xrightarrow{h_2} \Omega S^{4n-1} \xrightarrow{\lambda_{k}} \Omega S^{2n}.\] The map $(\eta_{k})_*$ is trivial on $H_{2n-1}(\Omega S^{2n})$ and sends $\delta$ to $2k\delta$. We adjust the map $\rho_v$ by defining the map $\overline{\rho}_v$ as the composite \[\overline{\rho}_v:\Omega S^{2n} \xrightarrow{\Delta} \prod\limits_{i=1}^{m_{Y}} \Omega S^{2n} \times \prod\limits_{j=1}^{m_{\overline{Y}}} \Omega S^{2n} \xrightarrow{\prod\limits_{i=1}^{m_{Y}}\Delta \times \prod\limits_{j=1}^{m_{\overline{Y}}} h_2} \prod\limits_{i=1}^{m_{Y}} (\Omega S^{2n} \times \Omega S^{2n}) \times \prod\limits_{j=1}^{m_{\overline{Y}}} \Omega S^{4n-1}\]\[\xrightarrow{\prod\limits_{i=1}^{m_{Y}}(p_{y_i} \times \eta_{k_i'}) \times \prod\limits_{j=1}^{m_{\overline{Y}}} p_{\overline{y}_j+2\overline{k}_j} } \prod\limits_{i=1}^{m_{Y}} (\Omega S^{2n} \times \Omega S^{2n}) \times \prod\limits_{j=1}^{m_{\overline{Y}}} \Omega S^{4n-1} \xrightarrow{\prod\limits_{i=1}^{m_{Y}}\mu \times \prod\limits_{j=1}^{m_{\overline{Y}}} id} \prod\limits_{i=1}^{m_{Y}} \Omega S^{2n} \times \prod\limits_{j=1}^{m_{\overline{Y}}} \Omega S^{4n-1} \hookrightarrow X.\] By definition, $\overline{\rho}_v$ sends the generator $\gamma \in H_{2n-1}(\Omega S^{2n})$ to $v$, and the generator $\delta \in H_{4n-2}(\Omega S^{2n})$ to $\overline{w}$. 

  \begin{lemma}
 \label{firstadjustmenttoe}
 Suppose that $H_{2n-1}(A)$, $n \notin \{1,2,4\}$ contains a primitive generator. Then the composite \[\overline{e}:\Omega S^{2n} \xrightarrow{\overline{\rho}_v} X \xrightarrow{\phi} X \xrightarrow{\rho'_v} \Omega S^{2n}\] induces an isomorphism on $H_{2n-1}(\Omega S^{2n})$. Moreover, $(\overline{e})_*$ maps the generator $\delta \in H_{4n-2}(\Omega S^{2n})$ to an element of the form $(2\overline{k}'+1)\delta$ for some $\overline{k}' \in \mathbb{Z}$.
 \end{lemma}
 \begin{proof}
 The map $(\eta_{k})_*$ is trivial on $H$, and so $(\overline{\rho}_v)_*$ maps the generator $\gamma \in H$ to $v$. The element $v$ is fixed by $\phi_*$. The map $(\rho'_v)_*$ sends $v$ to $\gamma \in H$. Therefore, $\overline{e}$ induces an isomorphism on $H_{2n-1}(\Omega S^{2n})$. By Theorem ~\ref{atomicityoflocalisedloopsphere}, the map $\overline{e}$ is a $2$-local homotopy equivalence. Hence, $(\overline{e})_*$ maps $\delta \in H_{4n-2}(\Omega S^{2n})$ to an element of the form $(2\overline{k}+1)\delta$ for some $\overline{k} \in \mathbb{Z}$.
 \end{proof}

 Now we adjust $\rho'_v$ to obtain an isomorphism on the bottom two non-vanishing degrees in homology. Recall $\overline{w} = (y_1+2k'_1,\cdots, y_{m_{Y}}+2k'_{m_Y},\overline{y}_1+2\overline{k}_1,\cdots,\overline{y}_{m_{\overline{Y}}}+2\overline{k}_{m_{\overline{Y}}})$, and the greatest common divisor of the non-zero components of $\overline{w}$ is $1$. Therefore, by B\'ezout's Lemma, for $1 \leq i \leq m_Y$, and $1 \leq j \leq m_{\overline{Y}}$ there exist $c_i',\overline{c}_j \in \mathbb{Z}$ such that $\sum_{i=1}^{m_Y} c_i' (y_i +2k_i') + \sum_{j=1}^{m_{\overline{Y}}} \overline{c}_j (\overline{y}_j +2\overline{k}_j) = 1$. Let $\lambda'$ be the composite \[\lambda':\prod\limits_{i=1}^{m_Y} \Omega S^{2n} \times \prod\limits_{j=1}^{m_{\overline{Y}}} \Omega S^{4n-1} \xrightarrow{\prod\limits_{i=1}^{m_Y} h_2 \times \prod\limits_{j=1}^{m_{\overline{Y}}} id} \prod\limits_{i=1}^{m_Y} \Omega S^{4n-1} \times \prod\limits_{j=1}^{m_{\overline{Y}}} \Omega S^{4n-1}  \]\[\xrightarrow{\prod\limits_{i=1}^{m_Y} p_{c_i'} \times \prod\limits_{j=1}^{m_{\overline{Y}}} p_{\overline{c}_j} } \prod\limits_{i=1}^{m_Y} \Omega S^{4n-1} \times \prod\limits_{j=1}^{m_{\overline{Y}}} \Omega S^{4n-1}  \xrightarrow{\mu} \Omega S^{4n-1} \xrightarrow{ \lambda_{-\overline{k}'}} \Omega S^{2n}.\] By definition, $\lambda'$ sends the element $\overline{w}$ in $H_{4n-2}\left(\prod_{i=1}^{m_Y} \Omega S^{2n}_i \times \prod_{j=1}^{m_{\overline{Y}}} \Omega S^{4n-1}_j\right)$ to $-2\overline{k}\delta$, and is trivial in degree $2n-1$. Now adjust the map $\rho'_v$ by defining a map $\overline{\rho}'_v$ as the composite \[\overline{\rho}'_v:X \xrightarrow{\pi} \prod\limits_{i=1}^{m_Y} \Omega S^{2n}_i \times \prod\limits_{j=1}^{m_{\overline{Y}}} \Omega S^{4n-1}_j \xrightarrow{\Delta} \left(\prod\limits_{i=1}^{m_Y} \Omega S^{2n}_i \times \prod\limits_{j=1}^{m_{\overline{Y}}} \Omega S^{4n-1}_j\right) \times \left(\prod\limits_{i=1}^{m_Y} \Omega S^{2n}_i \times \prod\limits_{j=1}^{m_{\overline{Y}}} \Omega S^{4n-1}_j\right) \]\[\xrightarrow{\pi \times \lambda'} \left(\prod\limits_{i=1}^{m_Y} \Omega S^{2n}_i\right) \times \Omega S^{2n} \xrightarrow{\rho'_v \times id} \Omega S^{2n} \times \Omega S^{2n} \xrightarrow{\mu} \Omega S^{2n}.\]  Using $\overline{\rho}_v$ and $\overline{\rho}'_v$, we can now conclude that $\Omega S^{2n}$ retracts off $A$ when $Y$.

\begin{lemma}
Suppose that $H_{2n-1}(A)$ contains a primitive generator where $n \notin \{1,2,4\}$. Then $\Omega S^{2n}$ retracts off $A$.    
\end{lemma}
\begin{proof}
By definition of $\overline{\rho}_v$, the induced map $(\overline{\rho}_v)_*$, sends the generator $\gamma \in H_{2n-1}(\Omega S^{2n})$ to $v$, and the generator $\delta \in H_{4n-2}(\Omega S^{2n})$ to the element $\overline{w}$. By construction, the induced map $\overline{\rho}'_v$ maps the element $v$ to the generator $\gamma \in H_{2n-1}(\Omega S^{2n})$, and maps the element $\overline{w}$ to the generator $\delta \in H_{4n-2}(\Omega S^{2n})$.

Therefore since $v$ and $\overline{w}$ are fixed by $\phi_*$, the composite \[\overline{e}':\Omega S^{2n} \xrightarrow{\overline{\rho}_v} X \xrightarrow{\phi} X \xrightarrow{\overline{\rho}'_v} \Omega S^{2n}\] induces an isomorphism on $H_{2n-1}(\Omega S^{2n})$ and $H_{4n-2}(\Omega S^{2n})$. By Theorem ~\ref{atomicityoflocalisedloopsphere} at the prime $2$, the map $\overline{e}'$ is a $2$-local homotopy equivalence. Therefore, $\overline{e}'$ is also a rational homotopy equivalence. The splitting of $\Omega S^{2n}$ in Theorem ~\ref{oddprimessplitting} and atomicity of loops on odd spheres localised at an odd prime in Theorem ~\ref{atomicityoflocalisedloopsphere} implies that $\overline{e}'$ is a homotopy equivalence when localised at any odd prime. Since $\overline{e}'$ is a homotopy equivalence localised at every prime and rationally, $\overline{e}'$ is an integral homotopy equivalence. Therefore, since $\phi$ factors through $A$, $\Omega S^{2n}$ retracts off $A$. 
\end{proof}

Now arguing as in Proposition ~\ref{spherescase}, we obtain the following result.

\begin{proposition}
\label{loopspheresevencase}
    Let $X \in \mathcal{P}$ and $A$ be a space which retracts off $X$. Suppose that $H_{2n-1}(A)$ contains a primitive generator where $n \notin \{1,2,4\}$. Then there is a homotopy equivalence \[A \simeq \Omega S^{2n} \times A'\] where $A'$ retracts off $X$ and $H_{2n-1}(A')$ contains one fewer primitive generator than $H_{2n-1}(A)$.
    \qedno
\end{proposition}

\subsection{Conclusion of proof}

We can combine the work of the previous sections to conclude that $\mathcal{P}$ is closed under retracts.

\begin{theorem}
\label{retractofPisinP}
Let $X \in \mathcal{P}$, and let $A$ be a space which retracts off $X$. Then $A \in \mathcal{P}$.
\end{theorem}
\begin{proof}
Let $n_0$ be the degree of the lowest non-trivial homology group of $A$, and let $k_0$ be the rank of $H_{n_0}(A)$. Observe that since $n_0$ is the lowest non-trivial degree, each of the primitive generators of $H_{n_0}(A)$ are in the Hurewicz image. Let $Y = S^{n_0}$ if $n_0 \in \{1,3,7\}$ or $Y = \Omega S^{n_0+1}$ otherwise. We claim that $A \simeq \prod\limits_{i=1}^{k_0} Y \times Z_0$ where $Z_0$ has no primitive generators in degree $n_0$, $Z_0$ retracts off $X$.
   
   We proceed by induction. Suppose $k_0 = 1$. Then by Proposition ~\ref{spherescase}, Proposition ~\ref{loopspheresmod3case} or Proposition ~\ref{loopspheresevencase} (depending on the parity of $n_0$), there is a homotopy equivalence \[A \simeq Y \times Z_0\] where $Z_0$ retracts off $X$, and $H_{n_0}(Z_0)$ contains one fewer primitive generator than $H_{n_0}(Z_0)$. Since $k_0 = 1$, $Z_0$ contains no primitive generators in degree $n_0$.

   Now suppose the result is true for $m-1$ and suppose $k_0 = m$. Then by Proposition ~\ref{spherescase}, Proposition ~\ref{loopspheresmod3case} or Proposition ~\ref{loopspheresevencase}, there is a homotopy equivalence \[A \simeq Y \times A'\] where $A'$ retracts off $X$, and $H_{n_0}(A')$ contains one fewer primitive generator than $H_{n_0}(A')$. Therefore, the inductive hypothesis implies that there is a homotopy equivalence $A' \simeq \prod_{i=1}^{m-1} Y \times Z_0$, where $Z_0$ retracts off $X$ and has no primitive generators in degree $n_0$. Hence $A \simeq\prod_{i=1}^{m} Y \times Z_0$ as claimed.

   Observe that $Z_0$ is more highly connected than $A$. Let $n_1$ be the degree of the lowest non-trivial homology group of $Z$, and let $k_1$ be the rank of $H_{n_1}(Z_0)$. We can repeat this argument to obtain a homotopy equivalence $Z_0 \simeq P \times Z_1$ where $P$ is a product of spheres or loops on spheres whose lowest non-trivial homology group is $n_1$, $Z_1$ retracts off $X$, and $Z_1$ is more highly connected than $Z_0$. Since $A$ is of finite type, we can iteratively repeat this argument for each degree of $H_*(A)$ containing a primitive generator. Therefore, we obtain that $A \in \mathcal{P}$. 
\end{proof}

\section{Loop space decompositions of pushouts of polyhedral products as a product of spheres and loops on spheres}
\label{sec:ApptoPP}

Recall from the introduction that $\mathcal{W}$ is the collection of topological spaces which are homotopy equivalent to a finite type wedge of simply connected spheres, and $\mathcal{P}$ is the collection of $H$-spaces which are homotopy equivalent to a finite type product of spheres and loops on simply connected spheres. The purpose of this section is to apply Theorem ~\ref{retractofPisinP} to prove that under mild hypotheses, if a simplicial complex $K$ can be decomposed as a pushout of simplicial complexes for which the loop space of the associated polyhedral product is in $\mathcal{P}$, then $\Omega \caa^K \in \mathcal{P}$.

\begin{theorem}
\label{pushoutofPisinP}
    Let $K$ be a simplicial complex defined as the pushout \[\begin{tikzcd}
	L & {K_1} \\
	{K_2} & K
	\arrow[from=1-1, to=1-2]
	\arrow[from=1-2, to=2-2]
	\arrow[from=1-1, to=2-1]
	\arrow[from=2-1, to=2-2]
\end{tikzcd}\] where either $L = \emptyset$ or $L$ is a proper full subcomplex of $K_1$ and $K_2$. If $\Sigma A_i \in \mathcal{W}$ for all $i$, $\Omega \caa^{K_1} \in \mathcal{P}$ and $\Omega \caa^{K_2} \in \mathcal{P}$, then $\Omega \caa^K \in \mathcal{P}$.
\end{theorem}
\begin{proof}
    If $L$ is the empty set, since $K = K_1 \cup_{L} K_2$, by Proposition ~\ref{emptysetdecomposition}, we obtain a homotopy equivalence \begin{equation}\label{emptysetcase}\Omega \caa^K \simeq \Omega((\mathcal{A} * \mathcal{A}') \vee (\caa^{K_1} \rtimes \mathcal{A}') \vee (\mathcal{A} \ltimes \caa^{K_2}))\end{equation} where $\mathcal{A}$ and $\mathcal{A}'$ are a product of $A_i$'s. If $L$ is a full subcomplex of $K_1$ and $K_2$, since $K = K_1 \cup_{L} K_2$, the simplicial complex $K$ satisfies the hypothesis of Proposition ~\ref{fullsubcomplexdecomp}, so there is a homotopy equivalence \begin{equation}\label{decompositiontouse}\Omega \caa^K \simeq \Omega \caa^L \times  \Omega\left((\mathcal{A} * \mathcal{A}') \vee (G \rtimes \mathcal{A}') \vee (\mathcal{A} \ltimes H)\right) \end{equation} where $\mathcal{A}$ and $\mathcal{A}'$ are a product of $A_i$'s, and $G$ and $H$ are the homotopy fibres of the retractions $f_1:\caa^{K_1} \rightarrow \caa^L$ and $f_2:\caa^{K_2} \rightarrow \caa^L$ respectively. Since $\Omega \caa^L$ retracts off $\Omega \caa^{K_1}$, Theorem ~\ref{retractofPisinP} implies that $\Omega \caa^{L} \in \mathcal{P}$. By Corollary ~\ref{loopofwedgeofP}, to show that the decompositions in (\ref{emptysetcase}) and (\ref{decompositiontouse}) are in $\mathcal{P}$, it suffices to show that each of $\Omega (\mathcal{A}*\mathcal{A}')$, $\Omega(\caa^{K_1} \rtimes \mathcal{A}')$, $\Omega(\mathcal{A} \ltimes \caa^{K_2})$ are in $\mathcal{P}$ for (\ref{emptysetcase}), and additionally $\Omega(G \rtimes \mathcal{A}')$ and $\Omega(\mathcal{A} \ltimes H)$ are in $\mathcal{P}$ for (\ref{decompositiontouse}).

    Since $\mathcal{A}$ and $\mathcal{A}'$ are products of $A$'s and $\Sigma A \in \mathcal{W}$ by assumption, it follows that $\Sigma \mathcal{A} \in \mathcal{W}$, $\Sigma \mathcal{A}' \in \mathcal{W}$ and $\mathcal{A} *\mathcal{A}' \in \mathcal{W}$. Therefore, the Hilton-Milnor theorem implies that $\Omega (\mathcal{A} * \mathcal{A}') \in \mathcal{P}$. Since $\Omega \caa^{K_1} \in \mathcal{P}$ and $\Omega \caa^{K_2} \in \mathcal{P}$ by hypothesis, and $\Sigma \mathcal{A},\: \Sigma \mathcal{A}' \in \mathcal{W}$, by Lemma ~\ref{loophalfsmashinP}, $\Omega (\caa^{K_1} \rtimes \mathcal{A}') \in \mathcal{P}$ and $\Omega(\mathcal{A} \ltimes \caa^{K_2})\in \mathcal{P}$. Therefore, if $L$ is the empty set, then $\Omega \caa^K \in \mathcal{P}$.
    
    Now consider $G \rtimes \mathcal{A}'$. By Lemma ~\ref{loophalfsmashinP}, to show $\Omega (G \rtimes \mathcal{A}') \in \mathcal{P}$, it suffices to show that $\Omega G \in \mathcal{P}$. The map $f_1: \caa^{K_1} \rightarrow \caa^{L}$ has a right homotopy inverse, which implies that there is a homotopy equivalence $\Omega \caa^{K_1} \simeq \Omega \caa^L \times \Omega G$. Therefore $\Omega G$ retracts off $\Omega \caa^{K_1}$. Since $\Omega \caa^{K_1} \in \mathcal{P}$ by hypothesis, Theorem ~\ref{retractofPisinP} implies $\Omega G \in \mathcal{P}$. A similar argument shows that $\Omega (\mathcal{A} \ltimes H) \in \mathcal{P}$. Hence $\Omega \caa^K \in \mathcal{P}$.
\end{proof}

The next result will be used to show that if $K$ is the $k$-skeleton of a simplex, then the loop space of certain polyhedral products is in $\mathcal{P}$. The following result was first proved by Porter \cite[Theorem 1]{P} in the $(\underline{C \Omega A},\underline{\Omega A})^K$ case, and was generalised independently by \cite[Theorem 1.1]{GT1} and \cite[Theorem 1.7]{IK1} for general polyhedral products of the form $\caa^K$.

\begin{proposition}
\label{shifteddecomp}
    Let $K$ be the $k$-skeleton of $\Delta^{m-1}$. Then there is a homotopy equivalence \[\caa^K \simeq \bigvee\limits_{j=k+2}^m\left(\bigvee\limits_{1 \leq i_1 < \cdots < i_j \leq m} (\Sigma^{k+1} A_{i_1} \wedge \cdots \wedge A_{i_j})^{\vee \binom{j-1}{k+1}}\right).\]
    \qedno
\end{proposition}

This proposition can be used to prove the following lemma.

\begin{lemma}
\label{shiftedinP}
    Let $K$ be the $k$-skeleton of $\Delta^{m-1}$ and $A_1,\cdots,A_m$ be spaces such that $\Sigma A_i \in \mathcal{W}$ for all $i$, then $\caa^K \in \mathcal{W}$.
\end{lemma}

\begin{proof}
Since $\Sigma A_i \in \mathcal{W}$ for all $i$, by shifting the suspension coordinate it follows that $\Sigma^{k+1} A_{i_1} \wedge \cdots \wedge A_{i_j} \in \mathcal{W}$.
\end{proof}

For a general simplicial complex $K$, a general decomposition of $K$ will be required in order to apply Theorem ~\ref{pushoutofPisinP}. Let $V(K)$ be the vertex set of $K$ and for a subset $S \subseteq V(K)$, let $K_{S}$ be the full subcomplex of $K$ on the vertices of $S$. For a vertex $v \in V(K)$, denote by $N(v)$ the set of vertices adjacent to $v$ in the $1$-skeleton of $K$.

\begin{lemma}
\label{decomposesimpcomp}
    Let $K$ be a simplicial complex and $v \in V(K)$. Then $K$ can be written as the pushout \[\begin{tikzcd}
	{K_{N(v)}} & {K_{v \cup N(v)}} \\
	{K_{V(K) \setminus \{v\}}} & K.
	\arrow[from=2-1, to=2-2]
	\arrow[from=1-2, to=2-2]
	\arrow[from=1-1, to=1-2]
	\arrow[from=1-1, to=2-1]
\end{tikzcd}\] Moreover, $K_{N(v)}$ is a full subcomplex of both $K_{v \cup N(v)}$ and $K_{V(K) \setminus \{v\}}$.
\end{lemma}
\begin{proof}
    Since $K_{V(K) \setminus \{v\}}$ contains every simplex which does not contain the vertex $v$ and $K_{v \cup N(v)}$ contains every simplex containing $v$, $K_{v \cup N(v)} \cup K_{V(K) \setminus \{v\}} = K$.

    Clearly, $K_{N(v)} \subseteq K_{v \cup N(v)} \cap K_{V(K) \setminus \{v\}}$, so let $\sigma \in K_{v \cup N(v)} \cap K_{V(K) \setminus \{v\}}$. Since $\sigma \in K_{v \cup N(v)}$, $\sigma$ must have vertices in $v \cup N(v)$. However since $\sigma \in K_{V(K) \setminus \{v\}}$, none of the vertices can be $v$. Hence $K_{v \cup N(v)} \cap K_{V(K) \setminus \{v\}} \subseteq K_{N(v)}$, and so $K_{v \cup N(v)} \cap K_{V(K) \setminus \{v\}} = K_{N(v)}$.

    By definition, the subcomplex $K_{N(v)}$ contains every simplex in $K$ on the vertex set $N(v)$. Therefore, it is a full subcomplex of both $K_{v \cup N(v)}$ and $K_{V(K) \setminus \{v\}}$.
    \end{proof}

 We now prove Theorem ~\ref{loopofkskelinP}. Recall that $k \geq 0$, and let $K$ be the $k$-skeleton of a flag complex on the vertex set $[m]$. Let $A_1,\cdots,A_m$ be path connected $CW$-complexes such that $\Sigma A_i \in \mathcal{W}$ for all $i$. Then we wish to prove that $\Omega \caa^{K} \in \mathcal{P}$. A \textit{dominating vertex} of $K$ is a vertex $v$ such that $N(v) = V(K) \setminus \{v\}$. In other words, $v$ is adjacent to every other vertex in the $1$-skeleton of $K$.

\begin{proof}[Proof of Theorem ~\ref{loopofkskelinP}]
    We proceed by strong induction. If $K$ has one vertex, then $\cxx^{K}$ is contractible, and so $\Omega \cxx^{K} \in \mathcal{P}$.

    Now suppose $K$ has $m$ vertices, and the result is true for all $n <m$. Since $K$ is the $k$-skeleton of a flag complex, if every vertex of $K$ is a dominating vertex, then $K$ is the $k$-skeleton of a simplex. In this case, Lemma ~\ref{shiftedinP} implies that $\Omega \caa^{K} \in \mathcal{P}$. Therefore, suppose there exists a vertex $v \in V(K)$ such that $v$ is not a dominating vertex of $K$. By Lemma ~\ref{decomposesimpcomp}, $K$ can be written as the pushout \[\begin{tikzcd}
	{K_{N(v)}} & {K_{v \cup N(v)}} \\
	{K_{V(K) \setminus \{v\}}} & K
	\arrow[from=2-1, to=2-2]
	\arrow[from=1-2, to=2-2]
	\arrow[from=1-1, to=1-2]
	\arrow[from=1-1, to=2-1]
\end{tikzcd}\] where $K_{N(v)}$ is a full subcomplex of both $K_{v \cup N(v)}$ and $K_{V(K) \setminus \{v\}}$. Since $v$ is not a dominating vertex, $K_{v \cup N(v)}$ is not the whole of $K$, and so $K_{v \cup N(v)}$ and $K_{V(K) \setminus \{v\}}$ are simplicial complexes with strictly fewer vertices than $K$. Therefore, by the inductive hypothesis, $\Omega \caa^{K_{v \cup N(v)}} \in \mathcal{P}$ and $\Omega \caa^{K_{V(K) \setminus \{v\}}} \in \mathcal{P}$. Hence, Theorem ~\ref{pushoutofPisinP} implies that $\Omega \caa^{K} \in \mathcal{P}$. 
\end{proof}

\begin{remark}
    In principle, one could iteratively use Proposition ~\ref{emptysetdecomposition}, Proposition ~\ref{fullsubcomplexdecomp} and Lemma ~\ref{decomposesimpcomp} to obtain an explicit decomposition for $\cxx^{K}$. However, in practice, this process would be unwieldy. 
\end{remark}

Theorem ~\ref{loopofkskelinP} also has consequences for other polyhedral products associated to the $k$-skeleton of flag complexes.

\begin{lemma}
\label{fibrepolyinP}
    Let $K$ be a simplicial complex on the vertex set $[m]$ and let $(\underline{X},\underline{A})$ be any sequence of pointed, path-connected $CW$-pairs. Denote by $Y_i$ the homotopy fibre of the inclusion $A_i \rightarrow X_i$. Suppose $\Omega \cyy^K \in \mathcal{P}$ and for $1 \leq i \leq m$,  $\Omega X_i \in \mathcal{P}$ for all $i$. Then $\Omega \uxa^K \in \mathcal{P}$.
\end{lemma}
\begin{proof}
    By \cite[Theorem 2.1]{HST}, there is a homotopy fibration \[\cyy^K \rightarrow \uxa^K \rightarrow \prod\limits_{i=1}^m X_i\] which splits after looping. Therefore, there is a homotopy equivalence \[\Omega \uxa^K \simeq \prod\limits_{i=1}^m \Omega X_i \times \Omega \cyy^K.\] By assumption, $\Omega X_i \in \mathcal{P}$ for all $i$ and $\Omega \cyy^K \in \mathcal{P}$, and so $\Omega \uxa^K \in \mathcal{P}$.
\end{proof}

When $K$ is the $k$-skeleton of a flag complex, Theorem ~\ref{loopofkskelinP} and Lemma ~\ref{fibrepolyinP} implies the following result.

\begin{corollary}
\label{generalfibgraphinP}
    Let $K$ be the $k$-skeleton of a flag complex on the vertex set $[m]$. Let $(\underline{X},\underline{A})$ be any sequence of pointed, path-connected $CW$-pairs, and denote by $Y_i$ the homotopy fibre of the inclusion $A_i \hookrightarrow X_i$. Suppose $\Omega X_i \in \mathcal{P}$ for all $i$ and $\Sigma Y_i \in \mathcal{W}$ for all $i$. Then $\Omega \uxa^{K} \in \mathcal{P}$.
\end{corollary}
\begin{proof}
    Since $\Sigma Y_i \in \mathcal{W}$ for all $i$, Theorem ~\ref{loopofkskelinP} implies that $\Omega \cyy^{K} \in \mathcal{P}$. By assumption $\Omega X_i \in \mathcal{P}$ for all $i$, so Lemma ~\ref{fibrepolyinP} implies that $\Omega \uxa^{K} \in \mathcal{P}$.
\end{proof}

Corollary ~\ref{generalfibgraphinP} applies to more examples, as follows.

\begin{corollary}
\label{complexprojectivePP}

Let $K$ be the $k$-skeleton of a flag complex on the vertex set $[m]$, and for $1 \leq i \leq m$, let $n_i \in \mathbb{N} \cup \{\infty\}$ and $m_i \in \mathbb{N}$ with $m_i < n_i$. Let $(X_i,A_i) = (\mathbb{C}P^{n_i},\mathbb{C}P^{m_i})$ or $(X_i,A_i) = (\mathbb{C}P^{n_i},*)$ for all $i$. Then $\Omega \uxa^{K} \in \mathcal{P}$.
\end{corollary}

\begin{proof}
    There are homotopy equivalences $\Omega \mathbb{C}P^{k} \simeq S^1 \times \Omega S^{2k+1}$ and $\Omega \mathbb{C}P^{\infty} \simeq S^1$. Therefore, $\Omega \mathbb{C}P^{n_i} \in \mathcal{P}$ for all $i$. First consider a pair of the form $(\mathbb{C}P^{n_i},*)$. The homotopy fibre $Y_i$ of the inclusion of the basepoint into $\mathbb{C}P^{n_i}$ is $\Omega \mathbb{C}P^{n_i}$, and so $Y_i \in \mathcal{P}$. Hence, $\Sigma Y_i \in \mathcal{W}$. 

    Now consider a pair of the form $(\mathbb{C}P^{n_i},\mathbb{C}P^{m_i})$. Suppose $n_i = \infty$. In this case, there is a standard homotopy fibration \[S^{2k+1} \rightarrow \mathbb{C}P^k \rightarrow \mathbb{C}P^\infty,\] and $\Sigma S^{2k+1} \in \mathcal{W}$. Now suppose $n_i \neq \infty$. Consider the homotopy fibration diagram \[\begin{tikzcd}
	{S^{2m+1} \times \Omega S^{2n+1}} & {S^{2m+1}} & {S^{2n+1}} \\
	{Y_i} & {\mathbb{C}P^{m_i}} & {\mathbb{C}P^{n_i}} \\
	& {\mathbb{C}P^\infty} & {\mathbb{C}P^\infty}
	\arrow[from=2-2, to=2-3]
	\arrow[from=2-3, to=3-3]
	\arrow[from=2-2, to=3-2]
	\arrow[Rightarrow, no head, from=3-2, to=3-3]
	\arrow[from=1-2, to=2-2]
	\arrow[from=1-3, to=2-3]
	\arrow["{*}", from=1-2, to=1-3]
	\arrow[from=2-1, to=2-2]
	\arrow[from=1-1, to=1-2]
	\arrow[from=1-1, to=2-1]
\end{tikzcd}\] where the maps in the bottom square are all inclusions, and the top map in the top right square is null homotopic since $m < n$. The top right square is a homotopy pullback, implying that $Y_i \simeq S^{2m+1} \times \Omega S^{2n+1}$. Therefore, $Y_i \in \mathcal{P}$, and so $\Sigma Y_i \in \mathcal{W}$. Hence, $\Sigma Y_i \in \mathcal{W}$ for all $i$, and Lemma ~\ref{generalfibgraphinP} implies that $\Omega \uxa^{K} \in \mathcal{P}$.
\end{proof}

%%% The bibliography %%%
\bibliographystyle{amsalpha}

\begin{thebibliography}{BBCG}

\bibitem[1]{Ad} J.F. Adams, On the non-existence of elements of Hopf invariant one, \textit{Ann. of Math.} \textbf{72} (1960), 20-104.

\bibitem[2]{An} D.J. Anick, Single loop space decompositions. \textit{Trans. Amer. Math. Soc.} \textbf{334} (1992), 929-940.

\bibitem[3]{BBC} A. Bahri, M. Bendersky and F.R. Cohen,
   Polyhedral Products and features of their homotopy theory, \emph{Handbook of Homotopy Theory} 103-144, CRC Press, Boca Raton, FL, (2020).  

\bibitem[4]{Ca} L. Cai, On the graph products of simplicial groups and connected Hopf algebras, arXiv:2306.16625.   

\bibitem[5]{CMN1} F.R Cohen, J.C. Moore, and J.A Neisendorfer, The double suspension and exponents of the homotopy groups of spheres. \textit{Ann. of Math. (2)} \textbf{110} (1979), 549-565.

\bibitem[6]{CMN2} F.R Cohen, J.C. Moore, and J.A Neisendorfer, Torsion in homotopy groups. \textit{Ann. of Math. (2)} \textbf{109} (1979), 121-168.

\bibitem[7]{CMN3} F.R Cohen, J.C. Moore, and J.A Neisendorfer, Exponents in homotopy theory. \textit{Algebraic topology and algebraic $K$ theory}, W. Browder, ed., Ann. of Math. Study \textbf{113}, Princeton University Press, 1987, 3-34.

\bibitem[8]{CPS} H.E.A. Campbell, F.P. Peterson, and P.S Selick, Self-maps of loop spaces. I., \emph{Trans. Amer. Math. Soc}, \textbf{293} (1986), no.1, 41-51.

\bibitem[9]{D} N. Dobrinskaya, Loops on polyhedral products and diagonal arrangements, arXiv:0901.2871.

\bibitem[10]{DS} G. Denham and A.I. Suciu, Moment-angle Complexes, Monomial Ideals and Massey Products, \emph{Pure Appl. Math. Q.} \textbf{3} (2007), 25-60. 

\bibitem[11]{F} E. Dror Farjoun, Cellular spaces, null spaces and homotopy localization, \emph{Lecture Notes in Mathematics}, \textbf{1622}, Berlin, Springer-Verlag, (1996).

\bibitem[12]{GPTW} J. Grbi\'{c}, T. Panov, S. Theriault, and J. Wu, The homotopy types of moment-angle complexes for flag complexes, \emph{Trans. Amer. Math. Soc} 
   \textbf{368} (2016), no. 9, 6663-6682. 
 
\bibitem[13]{GT1} J. Grbi\'{c}, and S. Theriault, The homotopy type of 
   the polyhedral product for shifted complexes, \emph{Adv. Math.} 
   \textbf{245} (2013), 690-715. 

\bibitem[14]{GT2} J. Grbi\'{c}, and S. Theriault, The homotopy type of the complement of a coordinate subspace arrangement, \emph{Topology} 
   \textbf{46} (2007), no. 4, 357-396. 
   
\bibitem[15]{H} P.J Hilton, On the homotopy groups of the union of spheres, \emph{J. Lond. Math. Soc.},\textbf{30} (1955),154-172. 

\bibitem[16]{HST} Y. Hao, Q. Sun and S. Theriault, Moore's conjecture for polyhedral products, \emph{Math. Proc. Cambridge Philos. Soc.} 
   \textbf{167} (2019), no. 1 23-33. 

\bibitem[17]{IK1} K. Iriye, and D. Kishimoto, Decompositions of polyhedral products for shifted complexes, \emph{Adv. Math.} 
   \textbf{245} (2013), 716-736. 

\bibitem[18]{IK2} K. Iriye, and D. Kishimoto, Fat-wedge filtration and decomposition of polyhedral products, \emph{Kyoto J. Math.} 
\textbf{59} (2019), no. 1, 1-51. 

\bibitem[19]{J} I.M. James, Reduced product spaces, \emph{Ann. of Math.} \textbf{62} (1955), 170-197.

\bibitem[20]{L} S. Lang, Introduction to linear algebra, Springer New York (1986).

\bibitem[21]{M} J. Milnor, On the construction F[K], \emph{Algebraic,topology, A student's guide} Cambridge Univ. Press, London (1972), 119-136.

\bibitem[22]{P} G.J. Porter, The homotopy groups of wedges of suspensions, \emph{Amer. J. Math} 
   \textbf{88} (1966), 655-663.

\bibitem[23]{PT} T. Panov, S. Theriault, The homotopy theory of polyhedral products associated with flag complexes, \emph{Compos. Math.} \textbf{155} (2019), no.1, 206-228.

\bibitem[24]{Se} J-P. Serre, Homologie singuli\`ere des espaces fibr\'es. Applications, \emph{Ann. of Math. (2)} \textbf{54} (1951), 425-505.

\bibitem[25]{St} L. Stanton, Loop space decompositions of highly symmetric spaces with applications to polyhedral products, \emph{Eur. J. Math.} \textbf{9} (2023), no.4, 104.

\bibitem[26]{T1} S. Theriault, Polyhedral products for connected sums of simplicial complexes, \emph{Proc. Steklov Inst. Math.} \textbf{317} (2022), 151–160.

\bibitem[27]{T2} S. Theriault, Polyhedral products for wheel graphs and their generalizations, accepted by \emph{Fields Institute Communications}.


\end{thebibliography}

\end{document}